\documentclass{amsart}
\usepackage{amssymb}
\usepackage{mathrsfs}
\usepackage{color}
\usepackage[colorlinks=true,citecolor=blue,linkcolor=blue]{hyperref}

\newcommand{\Rl}{\mathbb{R}}
\newcommand{\Cplx}{\mathbb{C}}
\newcommand{\Itgr}{\mathbb{Z}}
\newcommand{\Ntrl}{\mathbb{N}}
\newcommand{\Circ}{\mathbb{T}}

\newcommand{\sym}{\mathrm{sym}}
\newcommand{\A}{\mathcal{A}}
\newcommand{\Bc}{\mathcal{L}}
\newcommand{\Kc}{\mathcal{K}}
\newcommand{\Lc}{\mathcal{L}}
\newcommand{\SU}{\mathrm{SU}}
\newcommand{\SO}{\mathrm{SO}}
\newcommand{\diag}{\mathrm{diag}}
\newcommand{\Sp}{\mathrm{Sp}}
\newcommand{\Tr}{\mathrm{Tr}}
\newcommand{\Dc}{\mathcal{D}}
\newcommand{\Vol}{\mathrm{Vol}}
\newcommand{\Ec}{\mathcal{E}}
\newcommand{\Sb}{\mathbb{S}}

\newcommand{\tr}{\mathrm{tr}}

\theoremstyle{plain}
\newtheorem{thm}{Theorem}[section]

\newtheorem{rem}[thm]{Remark}
\newtheorem{lem}[thm]{Lemma}
\newtheorem{cor}[thm]{Corollary}
\newtheorem{defi}[thm]{Definition}

\newcommand{\hightlight}{}

\title{A $C^*$-algebraic approach to the principal symbol II.}
\author{Edward McDonald}
\author{Fedor Sukochev}
\author{Dmitriy Zanin}

\begin{document}

\maketitle
\begin{abstract}
    We introduce an abstract theory of the principal symbol mapping for pseudodifferential operators 
    extending the results of \cite{Dao1} and providing a simple
    algebraic approach to the theory of pseudodifferential operators in settings important in noncommutative geometry.
    We provide a variant of Connes' trace theorem which applies to certain noncommutative settings, with
    a minimum of technical preliminaries. Our approach allows us to consider in a operators with non-smooth symbols,
    and we demonstrate the power of our approach by extending Connes' trace theorem to operators with non-smooth symbols
    in three examples: the Lie group $\SU(2)$, noncommutative tori and Moyal planes.
\end{abstract}

\section{Introduction}
    This paper follows on from the work of the second two named authors in \cite{Dao1}, where a $C^*$-algebraic approach to the definition of the principal
    symbol map for order $0$ pseudodifferential operators was introduced. 
    
    The approach to the theory of pseudodifferential operators on $\Rl^d$ suggested in \cite{Dao1} is detailed as follows. One considers
    the $C^*$-algebras $L_\infty(\Rl^d)$ and $L_\infty(\Sb^{d-1})$ of essentially bounded functions on Euclidean space $\Rl^d$
    and the unit sphere $\Sb^{d-1}$ respectively. These algebras are represented on the Hilbert space $L_2(\Rl^d)$ via
    \begin{align*}
        (\pi_1(f)\xi)(t) &= f(t)\xi(t),\quad t \in \Rl^d\\
        (\pi_2(g)\xi) &= g\left(\frac{\nabla}{(-\Delta)^{1/2}}\right)\xi
    \end{align*}
    where $f \in L_\infty(\Rl^d)$, $g \in L_\infty(\Sb^{d-1})$, $\xi \in L_2(\Rl^d)$, $\nabla$ denotes the gradient operator and $\Delta$
    is the Laplacian. {\hightlight  Alternatively, $\pi_2(g)$ may be defined as Fourier multiplication on $L_2(\Rl^d)$ by the homogeneous function $t \mapsto g\left(\frac{t}{|t|}\right)$.}
    
    One then introduces a $C^*$-algebra $\Pi$, defined to be the $C^*$-subalgebra of the algebra of all bounded linear operators on $L_2(\Rl^d)$ generated by $\pi_1(L_\infty(\Rl^d))$
    and $\pi_2(L_\infty(\Sb^{d-1}))$. It is then proved (as \cite[Theorem 1.2]{Dao1}) that there is a unique norm-continuous $*$-homomorphism $\sym:\Pi\to L_\infty(\Rl^d\times \Sb^{d-1})$
    mapping $\pi_1(f)$ to $f\otimes 1$ and $\pi_2(g)$ to $1\otimes g$, for all $f \in L_\infty(\Rl^d)$ and all $g \in L_\infty(\Sb^{d-1})$.
    
    The algebra $\Pi$ contains (up to compact perturbations) all pseudodifferential operators of order $0$ on $\Rl^d$, and the mapping $\sym$ is an extension
    of the principal symbol mapping to non-smooth symbols. The motivation for introducing the algebra $\Pi$ was to develop in a self-contained manner the theory
    of order zero operators on $\Rl^d$, with a particular focus on Connes' trace formula \cite[Theorem 1]{Connes-Action}.

    In this paper we consider a far-reaching abstraction of the setting of \cite{Dao1}.    
    We consider two unital $C^*$-algebras $\A_1$ and $\A_2$ with faithful representations
    $\pi_1$ and $\pi_2$ respectively on the same Hilbert space $H$. 
    We define $\Pi(\A_1,\A_2)$ to be the $C^*$-subalgebra generated by $\pi_1(\A_1)$ and $\pi_2(\A_2)$. We discuss conditions on $\A_1$, $\A_2$, $\pi_1$ and $\pi_2$
    which allow the existence of a ``principal symbol map",
    \begin{equation*}
        \sym:\Pi(\A_1,\A_2) \to \A_1\otimes_{\min}\A_2
    \end{equation*}
    where $\otimes_{\min}$ denotes the minimal $C^*$-norm on the algebraic tensor product $\A_1\otimes \A_2$ (see Section \ref{C^* norms theory} for details
    on this tensor product norm). We verify these conditions in a number of important cases, in particular for noncommutative tori
    and noncommutative Euclidean spaces (also known as the Moyal plane or Moyal-Groenwald plane in the two-dimensional case).
    
    This framework is indeed an abstraction of the
    results of \cite{Dao1}, which considered the special case $\A_1 = L_\infty(\Rl^d)$, $\A_2 = L_\infty(\Sb^{d-1})$
    acting on $H = L_2(\Rl^d)$. However, the need to consider noncommutative algebras necessitated the introduction of new techniques, motivating the present paper.

    We show that under certain assumptions (given in Theorem \ref{symbol def}), the map $\sym$ is a $C^*$-algebra homomorphism, and {\hightlight  according to our constructions}
    $\A_1\otimes_{\min} \A_2$ is isometrically $*$-isomorphic to the image of $\Pi(\A_1,\A_2)$
    in the Calkin algebra $\Bc(H)/\Kc(H)$, where $\Bc(H)$ is the $*$-algebra of bounded
    operators on $H$, and $\Kc(H)$ denotes the ideal of compact operators.
    The principal symbol map $\sym$ is then simply the restriction to $\Pi(\A_1,\A_2)$ of the canonical quotient map $\Bc(H)\to \Bc(H)/\Kc(H)$ (see the proof
    of Theorem \ref{symbol def} for details).
    
    The idea that the principal symbol map can be identified with a Calkin quotient map is {\hightlight  far from }new. One particular utilisation of this idea in the commutative case
    appears in the work of H. O. Cordes \cite{Cordes-Elliptic}. In noncommutative geometry, this observation is
    crucial to the definition of the index map on $K$-theory and has been central to the development of tools such as the noncommutative geodesic flow (such as in \cite{Golse-Leichtnam-1998}). An example
    of the use of this idea in a noncommutative context is \cite[Section 2.1]{Nazaikinskii-Savin-Sternin-2008}.

    A novelty of our approach is its simplicity. With a minimum of algebraic preliminaries we are able to provide
    a version of the principal symbol mapping and Connes' trace theorem which applies to certain settings important for noncommutative geometry.
    The given proofs of our versions of Connes' Trace Theorem are totally different from earlier proofs, including
    Connes' original proof of the Connes' Trace Theorem for compact Riemannian manifolds \cite[Theorem 1]{Connes-Action} and the more recent proofs
    in \cite{Alberti-Matthes-2002}, \cite{Kalton-Lord-Potapov-Sukochev-2013} and \cite[Corollary 11.5.4]{LSZ}.

    To demonstrate the efficiency of our framework for pseudodifferential operators and the principal symbol mapping $\sym$, we develop in detail our theory for pseudodifferential operators of order $0$ on the Lie group $\SU(2)$ and
    give a version of Connes' Trace Theorem for $\SU(2)$ (Theorem \ref{ctt sphere}). As $\SU(2)$ is a compact Riemannian manifold, this result should be compared with previous versions of Connes' Trace Theorem (such
    as \cite[Theorem 1]{Connes-Action} and \cite[Theorem 11.6.17]{LSZ}). We emphasise
    that the proof given in this paper for Connes' Trace Theorem for $\SU(2)$ is completely different
    to previously known proofs, and we are able to present the argument in a self-contained way while also making
    no assumptions about smoothness of symbols which were used crucially in all previously known proofs.
    
    Beyond the commutative setting, we also work out in detail the two most heavily studied noncommutative spaces: the noncommutative tori
    and the noncommutative {\hightlight  (Moyal)} Euclidean spaces. {\hightlight 
    Noncommutative tori have long been heavily studied as a landmark example in noncommutative geometry \cite{Rieffel-contemp-math-1990}, and there is an extensive literature for pseudodifferential operators on
    these spaces. In
    \cite{Connes-C-star-algebras-and-differential-geometry}, Connes developed a pseudodifferential operator theory for general $C^*$-dynamical systems, and in particular for the special
    case of noncommutative tori. Baaj studied pseudodifferential calculus on crossed product $C^*$-algebras \cite{Baaj-1,Baaj-2}.
    Later, Connes \cite[Chapter 4, Section 6.$\alpha$]{red-book} studied elliptic differential operators on noncommutative tori. There
    is a diversity of related work concerning pseudodifferential calculus in $C^*$-algebraic settings, we mention in particular \cite{Mantoiu-Purice-Richard-2005,Lein-Mantoiu-Richard-2010}.
    In 2018, Ha, Lee and Ponge \cite{Ha-Lee-Ponge-I-arxiv,Ha-Lee-Ponge-II-arxiv} have given a very detailed
    exposition of pseudodifferential operator theory on noncommutative tori including many historical details.
      
    There is a closely related and well developed theory of harmonic analysis and function spaces on noncommutative tori. Early work includes the development of $L_2$-Sobolev spaces
    on noncommutative tori by Spera \cite{Spera-sobolev-spaces-1991,Spera-yang-mills-1992}. In more recent years, function space theory on noncommutative tori has been greatly extended by Xia, Xiong, Xu and Yin to
    analogues of many other classical function spaces, including the entire scale of Sobolev spaces \cite{XXX,XXY}.     
    
    The development of a pseudodifferential calculus for noncommutative Euclidean spaces has also been pursued by multiple authors: in particular 
    one can find such a theory in \cite[Proposition 4.17]{Gayral-Gracia-Bondia-Iochum-Schucker-Varily-moyal-planes-2004} and \cite[Section 5]{CGRS2}. 
    In 2017 a powerful pseudodifferential operator theory for noncommutative Euclidean spaces was developed by Gonzalez-Perez, Junge and Parcet \cite{Gonzalez-perez-Junge-Parcet-arxiv}.
    
    }
        
    {\hightlight 
        Our approach to pseudodifferential operators on noncommutative tori and planes is complementary to, and not a replacement for, the other approaches given in the above references. 
        In our approach we do not handle operators of order greater than zero and in fact the notion of order of a pseudodifferential operator does not play a role in the theory developed in this paper. The benefit of restricting
        to operators of order at most zero is that in this setting powerful operator algebraic techniques become available and the proof of Connes' trace theorem is simple and transparent. 
    }
    
    {\hightlight 
        We would like to extend our appreciation to Rapha\"el Ponge for extensive discussion of the results of this paper and suggestions relating to the historical overview of this field.
    }

%
\section{Preliminaries}

\subsection{Notations and conventions}
    We use the convention $\Ntrl = \{0,1,2,\ldots\}$, and all Hilbert spaces
    are over the field of complex numbers. 
    
    Given a Hilbert space $H$, $\Bc(H)$ denotes the $*$-algebra of all bounded linear operators on $H$,
    and $\Kc(H)$ denotes the ideal of compact operators. 
    
    The Schatten ideal $\Lc_p(H)$, for $p \in [1,\infty)$, is the set of $T \in \Kc(H)$
    whose sequence of singular values $\{\mu(n,T)\}_{n=0}^\infty$ is in the sequence space $\ell_p$.
    Similarly, the weak Schatten ideal $\Lc_{p,\infty}(H)$ is the set of compact operators 
    whose sequence of singular values is in the space $\ell_{p,\infty}$. {\hightlight  These spaces are a classical object of study, and further details may be found in \cite{Gohberg-Krein,Simon-trace-ideals-2005} and }\cite[Chapter 2]{LSZ}.
    
    For $p \in [1,\infty)$, $L_p(\Rl^d)$ denotes the usual Lebesgue
    space of pointwise-almost-everywhere equivalence classes of $p$-absolutely Lebesgue integrable
    functions, equipped with the $L_p$-norm.

\subsection{$C^*-$norms on tensor products of $C^*-$algebras}\label{C^* norms theory}

    Given two $C^*$-algebras $\A_1$ and $\A_2$, we denote the algebraic tensor product as $\A_1\otimes \A_2$.
    The following results are {\hightlight very well known, and are} taken from \cite{Sakai} (see Theorem 1.22.6, Propositions 1.22.5 and 1.22.3 there). {\hightlight  See also \cite[Chapter IV, Section 4]{Takesaki_1}, \cite[Section II.9]{Blackadar-OA} and
    \cite[Section 11.3]{Kadison-Ringrose-II}.}

    \begin{thm}\label{sakai1} Let $\mathcal{A}_1$ and $\mathcal{A}_2$ be unital $C^*-$algebras. There are pre-$C^*-$norms on the algebraic tensor product $\mathcal{A}_1\otimes\mathcal{A}_2,$ and there exists a norm which is minimal.
    \end{thm}

    The completion of $\mathcal{A}_1\otimes\mathcal{A}_2$ with respect to the minimal $C^*-$norm is denoted by $\mathcal{A}_1\otimes_{{\rm min}}\mathcal{A}_2.$

    \begin{thm}\label{sakai2} Let $\mathcal{A}_1$ and $\mathcal{A}_2$ be unital $C^*-$algebras. If $\mathcal{A}_2$ is commutative, then there exists a unique pre-$C^*-$norm on $\mathcal{A}_1\otimes\mathcal{A}_2$ (which we may take to be the minimal one).
    \end{thm}
    The above theorem is {\hightlight essentially a statement of the fact that commutative $C^*$-algebras are nuclear (see \cite[Section II.9.4]{Blackadar-OA} and \cite[Section 11.3]{Kadison-Ringrose-II})}.

    \begin{thm}\label{sakai3} 
        Let $\mathcal{A}_1$ and $\mathcal{A}_2$ be unital $C^*-$algebras. If $\mathcal{A}_2$ is commutative (read $\mathcal{A}_2=C(X)$ for some compact Hausdorff space $X$), then $\mathcal{A}_1\otimes_{{\rm min}}\mathcal{A}_2$ is
        isometrically $*$-isomorphic to $C(X,\mathcal{A}_1).$
    \end{thm}
    {\hightlight The above theorem immediately follows from Theorem \ref{sakai2} and the observation that $\A_1\otimes C(X)$ can be identified with a dense subalgebra of the $C^*$-algebra $C(X,\A_1)$.}

    {\hightlight  
    It is also well-known that the tensor product of continuous linear functionals is continuous on the tensor product algebra with minimal norm. We include the following simple proof for convenience. }
    \begin{thm} \label{linear functional tensor}
        Let $\A_1$ and $\A_2$ be $C^*$-algebras, and let $\psi_1 \in \A_1^*$ and $\psi_2 \in \A_2^*$. Then the tensor product $\psi_1\otimes\psi_2$ extends continuously
        to $\A_1\otimes_{\min} \A_2$.
    \end{thm}
    \begin{proof}
        First we may normalise $\|\psi_1\|_{\A_1^*} = \|\psi_2\|_{\A_2^*} = 1$. Then for $T \in \A_1\otimes \A_2$,
        \begin{equation*}
            |(\psi_1\otimes\psi_2)(T)| \leq \sup_{\alpha\otimes\beta \in \A_1^*\otimes \A_2^*, \|\alpha\|=\|\beta\|=1} |(\alpha\otimes\beta)(T)|.
        \end{equation*}
        The right hand side is the \emph{least cross norm} on the algebraic tensor product $\A_1\otimes \A_2$ denoted $\lambda(T)$, and from \cite[Proposition 1.22.2]{Sakai}, we have that $\lambda(T) \leq \|T\|_{\min}$. 
    \end{proof}

\subsection{Noncommutative torus}\label{nc torus definition section}
    Here, we present the definition and basic properties of the noncommutative $d$-torus. {\hightlight  There are many detailed expositions of this theory in the literature: we refer the reader in particular to \cite{Ha-Lee-Ponge-I-arxiv,XXY} and \cite[Chapter 12]{green-book}.}

    {\hightlight Let $d\geq 2$ be an integer, }and let $\theta$ be a real $d\times d$ antisymmetric matrix. The following defines the noncommutative $d$-torus,
    \begin{defi}\label{nc torus definition}
        The noncommutative $d$-torus $C(\Circ^d_\theta)$ is the universal $C^*$-algebra generated by a family of unitaries $\{u_n\}_{n \in \Itgr^d}$ satisfying the relation,
        \begin{equation*}
            u_{n}u_{m} = e^{\frac{i}{2}(n,\theta m)}u_{n+m},\,\,n,m \in \Itgr^d.
        \end{equation*}
        
        On the subspace given by the linear span of $\{u_n\}_{n \in \Itgr}$, we define the functional $\tau_\theta$,
        \begin{equation*}
            \tau_\theta\left(\sum_{k \in \Itgr^d} c_ku_k\right) := c_0.
        \end{equation*}
        The functional $\tau_\theta$ can be extended norm-continuously to all of $C(\Circ^d_\theta)$,
        and also to the enveloping von Neumann algebra which we denote $L^\infty(\Circ^d_\theta)$.
        We denote the corresponding $L_p$-spaces defined by $\tau_\theta$ as $L^p(\Circ^d_\theta)$.
        Note that in particular $L_2(\Circ^d_\theta)$ is the GNS-space for $L_\infty(\Circ^d_\theta)$ defined
        by $\tau_\theta$, with inner product:
        \begin{equation*}
            \langle x,y\rangle = \tau_\theta(x^*y).
        \end{equation*}
    \end{defi}

    \begin{defi}
        Let $j = 1,\ldots,d$. We define the operator $\partial_j$ on the linear span of $\{u_n\}_{n\in \Itgr^d}$, for each $k \in \Itgr^d$ by
        \begin{equation*}
            \partial_j (u_k) = ik_j u_k.
        \end{equation*}
        The operator $D_j := -i\partial_j$ extends to a self-adjoint densely defined operator on $L_2(\Circ^d_\theta)$. 
        Given a multi-index $\alpha \in \Ntrl^d$, we define $\nabla^\alpha := \partial_1^{\alpha_1}\partial_2^{\alpha_2}\cdots\partial_d^{\alpha_d}$.
        
        The gradient operator is defined as $\nabla := (D_1,D_2,\ldots,D_d)$, which may be considered as a linear operator from $L_2(\Circ^d_\theta)$ to $L_2(\Circ^d_\theta)\otimes\Cplx^d$.
    \end{defi}

\subsection{Noncommutative Euclidean space}\label{nc plane definition section}
    {\hightlight  Noncommutative Euclidean spaces appear in many different places in the literature under various names, such as Moyal product algebras \cite{Gayral-Gracia-Bondia-Iochum-Schucker-Varily-moyal-planes-2004}, quantum Euclidean spaces \cite{Gonzalez-perez-Junge-Parcet-arxiv} or canonical commutatation relation (CCR) algebras \cite[Section 5.2.2.2]{Bratteli-Robinson2}. 
    A detailed exposition of the operator theoretic details may be found in \cite[Section 5.2.2.2]{Bratteli-Robinson2},
    an exposition from the noncommutative geometric point of view is in \cite{Gayral-Gracia-Bondia-Iochum-Schucker-Varily-moyal-planes-2004} and more recently there is also \cite{Gonzalez-perez-Junge-Parcet-arxiv}
    from a harmonic analysis perspective. 
    
    The exposition we use here follows \cite{LeSZ_cwikel} although we skip many proofs. For details and proofs, we refer the reader to \cite[Section 6.1]{LeSZ_cwikel}.}

    For this section, {\hightlight we again take an integer $d \geq 2$} and define $\theta$ to be a non-degenerate $d\times d$ real antisymmetric matrix\footnote{Note that if $\det(\theta) \neq 0$ then $d$ is automatically even}. Our
    approach is to define the Noncommutative Euclidean space (also known as the Moyal plane)
    in terms of a certain family of unitary operators $\{U(t)\}_{t \in \Rl^d}$ {\hightlight  on the Hilbert space $L_2(\Rl^d)$.}

    \begin{defi}\label{nc plane definition}
        Let $t \in \Rl^d$. We define the following linear operator on $L_2(\Rl^d)$,
        \begin{equation*}
            (U(t)\xi)(u) = e^{-\frac{i}{2}(t,\theta u)}\xi(u-t),\quad {\hightlight  \xi \in L_2(\Rl^d).}
        \end{equation*}
        Then $\{U(t)\}_{t \in \Rl^d}$ is a strongly continuous family of unitary operators satisfying
        \begin{equation}\label{canonical commutation relations}
            U(t)U(s) = e^{\frac{i}{2}(t,\theta s)}U(t+s).
        \end{equation}
        The algebra $L_\infty(\Rl^d_\theta)$ is then defined to be the von Neumann algebra generated by $\{U(t)\}_{t \in \Rl^d}$. 
        
        Denote the representation of $L_\infty(\Rl^d_\theta)$ on $L_2(\Rl^d)$ as $\pi_1$\footnote{$\pi_1$ is simply the identity mapping}.
                
        It is known (see \cite{SZ_cif}) that there is an isometric $*$-isomorphism from $\Bc(L_2(\Rl^{d/2}))\to L_\infty(\Rl^d_\theta)$.
        Denote the image of the compact operators $\Kc(L_2(\Rl^{d/2}))$ under this isomorphism
        $C_0(\Rl^d_\theta)$.
        The standard trace $\Tr$ on $\Bc(L_2(\Rl^{d/2}))$ then induces a semifinite trace on the algebra $L_\infty(\Rl^d_\theta)$, which we denote as $\tau_\theta$.
        
        Finally, let $L_2(\Rl^d_\theta)$ be defined as the GNS space of $L_\infty(\Rl^d_\theta)$ with respect to $\tau_\theta$.
    \end{defi}

    \begin{rem}
        We also note that if we formally take $\theta = 0$ in Definition \ref{nc plane definition} we recover the commutative algebra $L_\infty(\Rl^d)$. However
        our definitions of $\tau_\theta$ and $C_0(\Rl^d_\theta)$ rely on the nondegeneracy of $\theta$. {\hightlight  A unified exposition for all $\theta$ is also possible, and this is achieved
        in \cite{Gonzalez-perez-Junge-Parcet-arxiv}.}
    \end{rem}
    
    {\hightlight  With our chosen concrete representation of $L_\infty(\Rl^d_\theta)$, defining the differentiation operators $\partial_1,\ldots,\partial_d$
    is particularly simple.}
    \begin{defi}
        For $k = 1,\ldots,n$, let $\partial_k$ denote the multiplication operators on $L_2(\Rl^d)$,
        \begin{equation*}
            D_k\xi(t) = t_k\xi(t).
        \end{equation*}
        We define the operators $\partial_kx$, $k=1,\ldots,d$ by
        \begin{equation*}
            \partial_kx := i[D_k,x].
        \end{equation*}
        There exists a dense subspace $\Dc \subset L_2(\Rl^d_\theta)$ such that the operators $\partial_k, k = 1,\ldots,d$ may be considered as 
        self-adjoint operators on $L_2(\Rl^d_\theta)$ with common core $\Dc$.
        We denote $\nabla = (\partial_1,\partial_2,\ldots,\partial_d)$, considered as a self-adjoint linear operator from $L_2(\Rl^d_\theta)$ to $L_2(\Rl^d_\theta)\otimes\Cplx^d$. For a multi-index $\alpha$,
        define
        \begin{equation*}
            \nabla^{\alpha} := \partial_1^{\alpha_1}\partial_2^{\alpha_2}\cdots \partial_d^{\alpha_d}
        \end{equation*}
        which is also considered as a self-adjoint operator on $L^2(\Rl^d_\theta)$.
        
        {\hightlight  For an essentially bounded function $g \in L_\infty(\Rl^d)$, we denote $g(\nabla)$ for the operator
        \begin{equation*}
            (g(\nabla)\xi)(t) = g(t)\xi(t),\quad t \in \Rl^d.
        \end{equation*}
        }
    \end{defi}
    {\hightlight  One should be careful to distinguish $g(\nabla)$ as defined above and the Fourier multiplier $g(\nabla)$ defined in the commutative situation, as in the introduction. If we formally take $\theta = 0$, then $L_\infty(\Rl^d_\theta)$ acts as Fourier multiplication on $L_2(\Rl^d)$, and so these definitions are ``Fourier dual" to each other.}
    
    \begin{defi}
        With $\tau_\theta$ we can define $L^p$-spaces associated to $L_\infty(\Rl^d_\theta)$ with the norm:
        \begin{equation*}
            \|x\|_{p} := \tau_\theta(|x|^p)^{1/p},\quad x \in L_\infty(\Rl^d_\theta).
        \end{equation*}
        Then denote $L_p(\Rl^d_\theta)$ for the corresponding $L_p$-space.
        Note that this is consistent with our definition of $L_2(\Rl^d_\theta)$ as a GNS-space.
        
        The corresponding Sobolev space, $W^{k,p}(\Rl^d_\theta)$ is defined to be the set of $x \in L^p(\Rl^d_\theta)$
        with $\nabla^{\alpha} x \in L^p(\Rl^d_\theta)$ for all $|\alpha| \leq k$. The $W^{k,p}$ norm
        is the sum of the $L_p$ norms of $\nabla^{\alpha}x$ for all $0 \leq |\alpha| \leq k$.
    \end{defi}
    {\hightlight  Sobolev spaces for noncommutative Euclidean spaces are also defined in \cite[Section 3.2.2]{Gonzalez-perez-Junge-Parcet-arxiv}.}

\subsubsection{Cwikel-type estimates for Noncommutative Euclidean Space}
    The following is \cite[Proposition 6.15({\it v})]{LeSZ_cwikel},
    \begin{lem}\label{ncp density lemma} $W^{m,2}(\mathbb{R}^d_{\theta})$ is a norm-dense subset of $C_0(\mathbb{R}^d_{\theta})$ for every $m\geq0.$
    \end{lem}
    We also require the following theorem, which is a special case of \cite[Theorem 7.2]{LeSZ_cwikel}:
    \begin{thm}\label{ncp cwikel lp}
        Let $p \in (2,\infty)$. If $x \in L_p(\Rl^d_\theta)$ and $g \in L_p(\Rl^d)$, then
        \begin{equation*}
            \|\pi_1(x)g(\nabla)\|_{\Lc_p} \leq C(p,d,\theta) \|x\|_{L_p(\Rl^d_\theta)} \|g\|_{L_p(\Rl^d)}.
        \end{equation*}
    \end{thm}

    The space $\ell_{1,\infty}(L_\infty)(\Rl^d)$
    is defined as the set of $g \in L_\infty(\Rl^d)$ such that:
    \begin{equation*}
        \left\{\mathrm{esssup}_{t \in n+[0,1]^d} |g(t)|\right\}_{n \in \Itgr^d} \in \ell_{1,\infty}(\Itgr^d).
    \end{equation*}
    The space $\ell_{1}(L_\infty)(\Rl^d)$ is defined similarly, with $\ell_1$ in place of $\ell_{1,\infty}$.

    The following is a special case of \cite[Theorem 7.6]{LeSZ_cwikel}:
    \begin{thm}\label{ncp cwikel} For every $x \in W^{d,1}(\Rl^d_\theta)$ and $g \in \ell_{1,\infty}(L_\infty(\Rl^d))$ we have that $\pi_1(x)g(\nabla) \in \Lc_{1,\infty}(L_2(\Rl^d))$
    and $\|\pi_1(x)g(\nabla)\|_{1,\infty} \leq C_{d,\theta}\|x\|_{W^{d,1}}\|g\|_{\ell_{1,\infty}(L_\infty)}$.
    \end{thm}

    Applying Theorem \ref{ncp cwikel} to the function $g(t) = (1+|t|^2)^{-d/2}$, we obtain a corollary,
    \begin{cor} \label{ncp specific cwikel}
        If $x \in W^{d,1}(\Rl^d_\theta)$, then
        \begin{equation*}
            \|\pi_1(x)(1-\Delta)^{-d/2}\|_{1,\infty} \leq C_d \|x\|_{W^{d,1}}.
        \end{equation*}
    \end{cor}

    The following is an $\Lc_1$ Cwikel estimate, proved in \cite[Theorem 7.7]{LeSZ_cwikel}
    \begin{lem}\label{ncp trace class cwikel}
        If $g \in \ell^{1}({L_\infty})(\Rl^d)$ and $x \in W^{d,1}(\Rl^d_\theta)$, then $\pi_1(x)g(\nabla) \in \Lc_1$.
    \end{lem}

\subsection{Geometry of ${\rm SU}(2)$}
    {\hightlight  Finally we discuss the relevant geometry of the Lie group $\SU(2)$. For a more thorough discussion, see \cite{ruzh-turu}.}
    We equip the group $\SU(2)$ with its unique right-invariant Haar measure $dg$.
    Let $\lambda_l,\lambda_r:{\rm SU}(2)\to \mathcal{L}(L_2({\rm SU}(2)))$ be the left and right regular representations given by the formulae,
    $$(\lambda_l(g_1)f)(g_2)=f(g_1^{-1}g_2),\quad f\in L_2({\rm SU}(2)),\quad g_1,g_2\in{\rm SU}(2),$$
    $$(\lambda_r(g_1)f)(g_2)=f(g_2g_1),\quad f\in L_2({\rm SU}(2)),\quad g_1,g_2\in{\rm SU}(2).$$
    {\hightlight Here, right invariance means that for all $h \in \SU(2)$,
    \begin{equation*}
        \int_{\SU(2)} (\lambda_r(h)f)(g)\,dg = \int_{\SU(2)} f(g)\,dg
    \end{equation*}
    for all $f \in L_1(\SU(2))$.}

    Recall that the Lie algebra $\mathfrak{su}(2)$ of the Lie group ${\rm SU}(2)$ is given by the formula 
    $$\mathfrak{su}(2)=\big\{x\in M_2(\mathbb{C}):\ x^*=-x,\ {\rm Tr}(x)=0\big\}.$$
    It is spanned by $i\sigma_k,$ $1\leq k\leq 3,$ where $\sigma_k,$ $1\leq k\leq 3,$ are the $2\times 2$ Pauli matrices:
    \begin{equation}\label{pauli matrices}
        \sigma_1 = \begin{pmatrix} 1 & 0\\0 & -1\end{pmatrix},\quad \sigma_2 = \begin{pmatrix} 0 & -i \\i & 0\end{pmatrix},\quad \sigma_3 = \begin{pmatrix} 0 & 1 \\ 1 & 0\end{pmatrix}.
    \end{equation}

    It is known that the Lie group ${\rm SU}(2)$ is generated by $1-$parameter subgroups $t\to\exp(it\sigma_k),$ $1\leq k\leq 3.$ Let $D_k,$ $1\leq k\leq 3,$ be the generator of the strongly continuous unitary group $t\to\lambda_l(\exp(it\sigma_k)).$

    We have
    $$[D_1,D_2]=2iD_3,\quad [D_2,D_3]=2iD_1,\quad [D_3,D_1]=2iD_2.$$
    Note that the Laplacian can be expressed by the formula
    $$-\Delta=D_1^2+D_2^2+D_3^2.$$
    It is immediate that
    $$[-\Delta,D_k]=0,\quad 1\leq k\leq 3.$$

    Hence, $\Delta$ commutes with each $\lambda_l(g)$, since it commutes with the generators
    of $\SU(2)$.
    It is known that $(1-\Delta)^{-3/2} \in \Lc_{1,\infty}$. Indeed, it is proved
    in \cite[Theorem 11.9.3]{ruzh-turu} that $\Delta$ has eigenvalues $\{-l(l+1)\}_{l \in \frac{1}{2}\Ntrl}$,
    where the $l$th eigenvalue has multiplicity $O(l^2)$.

    We consider an algebra (which we will denote $\mathcal{A}_2$) given as the $C^*-$subalgebra in $\mathcal{L}(L_2({\rm SU}(2)))$ generated by the operators
    $$b_k : =\frac{D_k}{(-\Delta)^{\frac12}},\quad 1\leq k\leq 3.$$
    Formally speaking the operators $D_k,$ $1\leq k\leq 3,$ and $-\Delta$ vanish on the subspace of constants. We define $b_k(1) = \frac{1}{\sqrt{3}}$
    so that $b_1^2+b_2^2+b_3^2 = 1$.

    Since $\Delta$ commutes with $D_k, k = 1,2,3$, it follows that:
    \begin{equation*}
        [b_j,b_k] = 2i(-\Delta)^{-1}D_l
    \end{equation*}
    for some $l$. Hence, $[b_j,b_k]$ is compact.

\section{Main construction}
    {\hightlight  In this section we describe the abstract algebraic framework for our principal symbol mapping. 
    The proofs of the fundamental results Lemma \ref{initial lemma} and Theorem \ref{symbol def} are straightforward, but instructive.
    For example, the proof of the following lemma follows almost immediately from Theorems \ref{sakai1}, \ref{sakai2} and \ref{sakai3}.}
    \begin{lem}\label{initial lemma} 
        Let $\mathcal{A}_1,$ $\mathcal{A}_2$ and $\mathcal{B}$ be $C^*-$algebras and let $\rho_1:\mathcal{A}_1\to\mathcal{B}$ and $\rho_2:\mathcal{A}_2\to\mathcal{B}$ be $C^*-$homomorphisms. Suppose that
        \begin{enumerate}
            \item\label{ina} $\mathcal{B}$ is generated by $\rho_1(\mathcal{A}_1)$ and $\rho_2(\mathcal{A}_2).$
            \item\label{inb} $\rho_1(x)$ commutes with $\rho_2(y)$ for all $x\in\mathcal{A}_1,$ $y\in\mathcal{A}_2.$
            \item\label{inc} $\mathcal{A}_1,\mathcal{A}_2$ are unital and $\mathcal{A}_2$ is abelian.
            \item\label{ind} The mapping $\theta:\mathcal{A}_1\otimes\mathcal{A}_2\to\mathcal{B}$ defined by the formula
            $$\theta(a_1\otimes a_2)=\rho_1(a_1)\rho_2(a_2),\quad a_1\in\mathcal{A}_1,\ a_2\in\mathcal{A}_2,$$
            is injective.
        \end{enumerate}
        Under these conditions, $\theta$ extends to a $C^*-$algebra isomorphism
        $$\theta:\mathcal{A}_1\otimes_{{\rm min}}\mathcal{A}_2\to\mathcal{B}.$$
    \end{lem}
    \begin{proof} It follows from condition \eqref{inb} that $\theta$ is a $*-$homomorphism. By condition \eqref{ind}, $\theta$ is an injection on the algebraic
        tensor product. This allows us to define a pre-$C^*-$norm on $\mathcal{A}_1\otimes\mathcal{A}_2$ by setting
        $$\|T\|=\|\theta(T)\|_{\mathcal{B}},\quad T\in\mathcal{A}_1\otimes\mathcal{A}_2.$$
        By condition \eqref{inc} and Theorem \ref{sakai2}, the latter norm coincides with the minimal pre-$C^*-$norm on $\mathcal{A}_1\otimes\mathcal{A}_2.$ Thus, $\theta:\mathcal{A}_1\otimes\mathcal{A}_2\to\mathcal{B}$ is an isometric embedding of the algebra $\mathcal{A}_1\otimes\mathcal{A}_2$ equipped with the minimal $C^*-$norm into $\mathcal{B}.$ Since $\mathcal{A}_1\otimes\mathcal{A}_2$ is dense in $\mathcal{A}_1\otimes_{{\rm min}}\mathcal{A}_2,$ the assertion follows from the condition \eqref{ina}.
    \end{proof}

    \begin{rem}
        Lemma \ref{initial lemma} uses the fact that there is a unique pre-$C^*$-norm on $\A_1\otimes \A_2$. It is enough to assume that one of the factors is nuclear, instead of abelian. For the remainder of this text
        we restrict to the case where one factor is abelian.
    \end{rem}

    Let $\mathcal{Q}(H)$ be the Calkin algebra and let $q:\mathcal{L}(H)\to\mathcal{Q}(H)$ be the quotient mapping.
    \begin{thm}\label{symbol def} Let $\mathcal{A}_1$ and $\mathcal{A}_2$ be $C^*-$algebras and let $\pi_1:\mathcal{A}_1\to\mathcal{L}(H)$ and $\pi_2:\mathcal{A}_2\to\mathcal{L}(H)$ be representations. Let $\Pi(\A_1,\A_2)$ be the $C^*-$algebra generated by $\pi_1(\mathcal{A}_1)$ and $\pi_2(\mathcal{A}_2).$ Suppose that
    \begin{enumerate}
    \item\label{symba} $\mathcal{A}_1,\mathcal{A}_2$ are unital and $\mathcal{A}_2$ is abelian.
    \item\label{symbb} The representations $\pi_1$ and $\pi_2$ ``commute modulo compact operators" i.e., for all $a_1 \in \A_1$ and $a_2 \in \A_2$ the commutator $[\pi_1(a_1),\pi_2(a_2)]$ is compact.
    \item\label{symbc} If $x_k\in\mathcal{A}_1,$ $y_k\in\mathcal{A}_2,$ $1\leq k\leq n,$ then
    $$\sum_{k=1}^n\pi_1(x_k)\pi_2(y_k)\in\mathcal{K}(H)\Longrightarrow\sum_{k=1}^nx_k\otimes y_k=0.$$
    \end{enumerate}
    There exists a unique continuous $*-$homomorphism $\sym:\Pi(\A_1,\A_2)\to\mathcal{A}_1\otimes_{{\rm min}}\mathcal{A}_2$ such that
    $$\sym(\pi_1(x))=x\otimes 1,\quad \sym(\pi_2(y))=1\otimes y,\quad x\in\mathcal{A}_1,y\in\mathcal{A}_2.$$
    \end{thm}
    \begin{proof} 
        This is a special case of Lemma \ref{initial lemma} with $\mathcal{B} = q(\Pi(\A_1,\A_2))$, and $\rho_j = q\circ \pi_j$, $j =1,2$. We verify
        each of the required conditions. Condition \ref{initial lemma}\eqref{ina} is satisfied since by definition $\mathcal{B} = q(\Pi(\A_1,\A_2))$ is generated by $\rho_1(\A_1)$ and $\rho_2(\A_2)$. Condition \ref{initial lemma}\eqref{inb} follows from \eqref{symbb}.    
        Condition \ref{initial lemma}\eqref{inc} is automatic, due to \eqref{symba}.
        
        Finally, condition \ref{initial lemma}\eqref{ind} is a consequence of \eqref{symbc}.
        
        Thus, Lemma \ref{initial lemma} states that
        \begin{equation*}
            \theta := \rho_1\otimes\rho_2
        \end{equation*}
        defines an isometric $*$-isomorphism $\theta:\A_1\otimes_{\min}\A_2 \to \mathcal{B}$.
        
        Define
        \begin{equation*}
            \sym := \theta^{-1}\circ q.
        \end{equation*}
        
        By construction $\sym:\Pi(\A_1,\A_2)\to \A_1\otimes_{\min} \A_2$ is a continuous $*$-algebra homomorphism. Let $x \in \A_1$. Then $\sym(\pi_1(x)) = \theta^{-1}(q(\pi_1(x)))$,
        and since $\theta(x\otimes 1) = \rho_1(x) = q(\pi_1(x))$, we get that $\sym(\pi_1(x)) = x\otimes 1$. Similarly, if $y \in \A_2$
        then $\sym(\pi_2(y)) = 1\otimes y$.
        
        As $\Pi(\A_1,\A_2)$ is generated by $\pi_1(\A_1)$ and $\pi_2(\A_2)$, and $\sym$ is continuous, it follows
        that $\sym$ is uniquely determined by its restriction to $\pi_1(\A_1)$ and $\pi_2(\A_2)$.    
    \end{proof}

    \begin{rem} 
        In Theorem \ref{symbol def}, it suffices to take $\pi_2(\mathcal{A}_2)$ abelian modulo compact operators. The target space of $\sym$ then becomes $\mathcal{A}_1\otimes_{{\rm min}}(q\circ\pi_2)(\mathcal{A}_2).$
    \end{rem}
    
    It is tempting to construct a symbol mapping in the following commutative cases:
    \begin{enumerate}
        \item $\mathcal{A}_1=C(\mathbb{T}^d),$ $\mathcal{A}_2=\ell_{\infty}(\mathbb{Z}^d),$ {\hightlight represented on $L_2(\Circ^d)$ by} $\pi_1(f)=M_f,$ $\pi_2(g)=g(\nabla).$
        \item $\mathcal{A}_1=C_b(\mathbb{R}^d),$ $\mathcal{A}_2=C_b(\mathbb{R}^d),$ {\hightlight represented on $L_2(\Rl^d)$ by} $\pi_1(f)=M_f,$ $\pi_2(g)=g(\nabla).$
    \end{enumerate}
    {\hightlight  In both cases, $M_f$ denotes pointwise multiplication by $f$ and $g(\nabla)$ denotes Fourier multiplication by $g$.}

    However, a simple lemma below shows this is impossible.
    \begin{lem}\label{nonhomogeneous commute lemma} 
        The representations $\pi_1$ and $\pi_2$ as above do not commute modulo compact operators.
    \end{lem}
    \begin{proof} 
        We consider the second case (the first one follows {\it mutatis mutandi}). Take $\alpha\in\mathbb{R}^d$ and set $g(t)=e^{i\langle t,\alpha\rangle},$ $t\in\mathbb{R}^d.$ We have
        $$([M_f,e^{i\langle \nabla,\alpha\rangle}]x)(t)=(f(t)-f(t+\alpha))x(t+\alpha),\quad t\in\mathbb{R}.$$
        Setting $h(t)=f(t)-f(t+\alpha),$ $t\in\mathbb{R},$ we obtain that
        $$[\pi_1(f),\pi_2(g)]=\pi_1(h)g(\nabla).$$
        Since $g(\nabla)$ is a unitary operator, it follows that compactness of $[\pi_1(f),\pi_2(g)]$ implies that of $\pi_1(h).$ The latter operator is compact if and only if $h=0.$ Hence, the commutator $[\pi_1(f),\pi_2(g)]$ fails to be compact (unless $f$ is $\alpha-$periodic).
    \end{proof}

    What is possible is to construct the symbol mapping for the following algebras.

    \begin{enumerate}
        \item $\mathcal{A}_1=C(\mathbb{T}^d),$ $\mathcal{A}_2=C(\Sb^{d-1}),$ $\pi_1(f)=M_f,$ $\pi_2(g)=g(\frac{\nabla}{(-\Delta)^{1/2}}).$
        \item $\mathcal{A}_1=\mathbb{C}+C_0(\mathbb{R}^d),$ $\mathcal{A}_2=C(\Sb^{d-1}),$ $\pi_1(f)=M_f,$ $\pi_2(g)=g(\frac{\nabla}{(-\Delta)^{1/2}}).$
    \end{enumerate}

    \begin{rem} If, in the second case, we take $\mathcal{A}_1=C_b(\mathbb{R}^d)$ {\hightlight (the algebra of bounded continuous functions on $\Rl^d$)} instead of $\mathcal{A}_1=\mathbb{C}+C_0(\mathbb{R}^d),$ then representations $\pi_1$ and $\pi_2$ would not commute modulo compact operators.
    \end{rem}
    \begin{proof} 
        Let $f(t) = e^{i(t,\alpha)}$ and let $g \in C(\Sb^{d-1})$ be non-constant. Denote the homogeneous extension of $g$ to $\Rl^d$
        as $\tilde{g}$. If $U$ denotes the Fourier transform, then
        \begin{equation*}
            U[\pi_1(f),\pi_2(g)]U^* = [e^{i(\alpha,\nabla)},M_{\tilde{g}}].
        \end{equation*}
        As already shown in Lemma \ref{nonhomogeneous commute lemma}, the above commutator is never compact since $\tilde{g}$ is never periodic.
    \end{proof}

    Furthermore, we are able to construct a symbol mapping for the following noncommutative algebras and representations:

    \begin{enumerate}
        \item $\mathcal{A}_1=C(\mathbb{T}^d_{\theta}),$ $\mathcal{A}_2=C(\Sb^{d-1}),$ $\pi_1(f)=M_f,$ $\pi_2(g)=g(\frac{\nabla}{(-\Delta)^{1/2}}).$
        \item $\mathcal{A}_1=\mathbb{C}+C_0(\mathbb{R}^d_{\theta}),$ $\mathcal{A}_2=C(\Sb^{d-1}),$ $\pi_1(f)=M_f,$ $\pi_2(g)=g(\frac{\nabla}{(-\Delta)^{1/2}}).$
    \end{enumerate}
    We also work with $\A_1 = C(\SU(2))$ and $\A_2$ is the $C^*$-algebra generated by the operators $b_1,b_2$ and $b_3$
    on $L_2(\SU(2))$. Here the $\A_2$ is noncommutative, however its image in the Calkin algebra, $q(\A_2)$ is commutative.

\section{Verification that the product mapping is injective}
    The most difficult part of verifying the conditions of Theorem \ref{symbol def} is \ref{symbol def}\eqref{symbc}. In this section
    we verify this condition in each of our examples.
\subsection{Noncommutative $d$-space}
    We start with the following basic fact:
    \begin{lem}\label{kdk lemma} If $T\in\mathcal{K}(H)$ and if $\{p_k\}_{k\geq0}\subset\mathcal{L}(H)$ is a sequence of pairwise orthogonal projections, then $\|Tp_k\|_{\infty}\to0$ as $k\to\infty.$
    \end{lem}
    \begin{proof} 
    Let $\varepsilon > 0$, and let $T = T_1+T_2$, where $T_1$ is finite rank and $\|T_2\|_{\infty} < \varepsilon$. 
    Since,
    \begin{equation*}
        \|Tp_k\|_\infty \leq \|T_1p_k\|_\infty + \varepsilon
    \end{equation*}
    it suffices to show that $\|T_1p_k\|_\infty\to 0$.

    Note that $\|T_1p_k\|_\infty \leq \|T_1p_k\|_2$. As each $p_k$ is pairwise orthogonal and $T_1 \in \mathcal{L}_2$,
    \begin{equation*}
        \sum_{k=0}^\infty \|T_1p_k\|_2^2 = \left\|T_1\sum_{k=0}^\infty p_k\right\|_2^2 < \infty.
    \end{equation*}
    Thus $\lim_{k\to\infty} \|T_1p_k\|_2 = 0$.
    %
    \end{proof}
    
    The following lemma verifies Condition \eqref{symbc} of Theorem \ref{symbol def} for the noncommutative $d$-space.
    \begin{lem}\label{ncp ver product} Let $x_k\in C_0(\mathbb{R}^d_{\theta})$ and $y_k\in C(\Sb^{d-1}),$ $1\leq k\leq n.$ If
    $$\sum_{k=1}^n\pi_1(x_k)\pi_2(y_k)\in \mathcal{K}(L_2(\mathbb{R}^d)),$$
    then
    $$\sum_{k=1}^nx_k\otimes y_k=0.$$
    \end{lem}
    \begin{proof} Fix $s\in \Sb^{d-1}$ and choose a sequence $\{m_j\}_{j\geq0}\subset\mathbb{Z}^d$ such that $\frac{m_j}{|m_j|}\to s$ and $|m_j|\to\infty$ as $j\to\infty.$ It follows that
    $$\sup_{t\in m_j+[0,1]^d}\left|\frac{t}{|t|}-s\right|\to0,\quad j\to\infty.$$
    By continuity, we have
    $$\sup_{t\in m_j+[0,1]^d}\left|y_k(\frac{t}{|t|})-y_k(s)\right|\to0,\quad j\to\infty.$$
    By the spectral theorem, we have
    $$\pi_2(y_k)\chi_{m_j+[0,1]^d}(\nabla)-y_k(s)\chi_{m_j+[0,1]^d}(\nabla)\to0$$
    in the uniform norm as $j\to\infty.$

    By Lemma \ref{kdk lemma}, we have that
    $$\sum_{k=1}^n\pi_1(x_k)\pi_2(y_k)\chi_{m_j+[0,1]^d}(\nabla)\to0$$
    in the uniform norm as $j\to\infty.$ By the preceding paragraph, we have
    $$\sum_{k=1}^n\pi_1(x_k)y_k(s)\chi_{m_j+[0,1]^d}(\nabla)\to0$$
    in the uniform norm as $j\to\infty.$ By Lemma \cite[Lemma 7.5]{LeSZ_cwikel}, there exists a unitary operator $V_j\in\mathcal{L}(L_2(\mathbb{R}^d))$ which commutes with $L_{\infty}(\mathbb{R}^d_{\theta})$ and such that
    $$V_j\chi_{m_j+[0,1]^d}(\nabla)V_j^{-1}=\chi_{[0,1]^d}(\nabla).$$
    Thus,
    $$\sum_{k=1}^n\pi_1(x_k)y_k(s)\chi_{[0,1]^d}(\nabla)=V\cdot(\sum_{k=1}^n\pi_1(x_k)y_k(s)\chi_{m_j+[0,1]^d}(\nabla))\cdot V^{-1}\to0$$
    in the uniform norm as $j\to\infty.$ The left hand side does not depend on $j$ and, therefore,
    $$\sum_{k=1}^n\pi_1(x_k)y_k(s)\chi_{[0,1]^d}(\nabla)=0.$$
    In other words, we have
    $$\pi_1(\sum_{k=1}^nx_k\cdot y_k(s))\cdot \chi_{[0,1]^d}(\nabla)=0.$$
    Thus,
    $$\sum_{k=1}^nx_k\cdot y_k(s)=0.$$
    Since $s\in \Sb^{d-1}$ is arbitrary, the assertion follows.
    \end{proof}

    The preceding Lemma applies for $\Rl^d_\theta$, where as always the assumption is made that $\det(\theta) \neq 0$.
    We also record the following, which applies for the commutative case $\Rl^d$:
    \begin{lem}\label{cp ver product} Let $x_k\in L_{\infty}(\mathbb{R}^d)$ and $y_k\in C(\Sb^{d-1}),$ $1\leq k\leq n.$ If
    $$\sum_{k=1}^n\pi_1(x_k)\pi_2(y_k)\in \mathcal{K}(L_2(\mathbb{R}^d)),$$
    then
    $$\sum_{k=1}^nx_k\otimes y_k=0.$$
    \end{lem}
    \begin{proof} The argument of Lemma \ref{ncp ver product} works \emph{mutatis mutandi} for this case, by taking
    instead $(V_j\xi)(t) := e^{-i(m_j,t)}\xi(t)$ rather than referring to \cite[Lemma 7.5]{LeSZ_cwikel}.
    \end{proof}

\subsection{Noncommutative $d$-tori}
    {\hightlight The canonical trace state $\tau_\theta$ on $C(\Circ^d_\theta)$ is invariant under the canonical $\Circ^d$-action (defined in Lemma \ref{nabla action lemma} below) which we denote $z_t$. Moreover, integrating $z_t$ over $\Circ^d$ recovers $\tau_\theta$.
    This fact is well known, and appears in \cite[Section 2]{Ha-Lee-Ponge-I-arxiv} and is implicit in \cite[Section 1.2]{XXY}.}
    \begin{lem}\label{nabla action lemma} 

    Let $x \in C(\Circ^d_\theta)$. We define the $1$-parameter family of operators,
    \begin{equation*}
        z_t(x) := e^{i(t,\nabla)}\pi_1(x)e^{-i(t,\nabla)}.
    \end{equation*}
    This family is continuous in the norm topology on $\Bc(L_2(\Circ^d_\theta))$,
    and furthermore we have
    \begin{equation}
    \label{translation integral}
        \int_{[-\pi,\pi]^d} z_t(x)\,dt = (2\pi)^d\tau_\theta(x)1
    \end{equation}
    where the left hand side is a Bochner integral in the norm topology of $\Bc(L_2(\Circ^d_\theta))$,
    and the $1$ on the right hand side is the identity operator.
    \end{lem}
    \begin{proof}
    One can compute,
    \begin{equation*}
        z_t(u_{n}) = e^{i(n,t)}\pi_1(u_n)
    \end{equation*}
    demonstrates the existence of the Bochner integral on the left hand side of \eqref{translation integral}.
    for $n \in \Itgr^d$. Hence $z_t(u_n)$ is norm-continuous in $t$, and if $x$ is a finite linear combination
    of $\{u_n\}_{n=0}^\infty$ then $z_t(x)$ is a linear combination of continuous functions and so is norm-continuous. Additionally,
    elements of $C(\Circ^d_\theta)$ are norm-limits of elements of the linear span of $\{u_n\}_{n\in \Itgr^d}$.
    Hence for all $x \in C(\Circ^d_\theta)$, $z_t(x)$ is a norm limit of continuous functions and so is continuous. This
    demonstrates the existence of the Bochner integral on the left hand side of \eqref{translation integral}.

    To prove the equality in \eqref{translation integral} it is enough to verify the result for $x = u_n$, since both
    sides are linear in $x$ and continuous in the norm topology. Since we have computed $z_t(u_n)$, this is straightforward.
    \end{proof}
    
    {\hightlight  The following lemma verifies Condition \eqref{symbc} in Theorem \ref{symbol def} for $C(\Circ^d_\theta)$.}
    \begin{lem} \label{torus injectivity}
        Let $x_k \in C(\Circ^d_\theta)$ and $y_k \in C(\Sb^{d-1})$, $1 \leq k \leq n$. If
        \begin{equation*}
            \sum_{k=1}^n \pi_1(x_k)\pi_2(y_k) \in \Kc(L_2(\Circ^d_\theta))
        \end{equation*}
        then
        \begin{equation*}
            \sum_{k=1}^n x_k\otimes y_k = 0.
        \end{equation*}
    \end{lem}
    \begin{proof} 
        If $\sum_{k=1}^n x_k\otimes y_k \neq 0$, we may
        assume without loss of generality that the set $\{y_k\}_{k=1}^n$
        is linearly independent. We may also assume, without loss, that there is at least one $k$
        with $\tau_\theta(x_k) \neq 0$ since if this is not the case, we may find $x$
        with $\tau_\theta(xx_k) \neq 0$ and consider instead the expression $\sum_{k=1}^n \pi_1(xx_k)\pi_2(y_k)$.
        
        We have that for each $t$
        \begin{equation*}
            e^{i(t,\nabla)}\left(\sum_{k=1}^n\pi_1(x_k)\pi_2(y_k)\right)e^{-i(t,\nabla)} \in \Kc(L_2(\Circ^d_\theta)),
        \end{equation*}
        and for each $k$, $e^{i(t,\nabla)}$ commutes with $\pi_2(y_k)$. Hence,
        \begin{equation*}
            \sum_{k=1}^n z_t(x_k)\pi_2(y_k) \in \Kc(L_2(\Circ^d_\theta)).
        \end{equation*}
        and from Lemma \ref{nabla action lemma}, this mapping is norm-continuous
        in the normed space $\Kc(L_2(\Circ^d_\theta))$. Hence,
        \begin{equation*}
            \sum_{k=1}^n\int_{[-\pi,\pi]^d} z_t(x_k)\,dt \cdot \pi_2(y_k) \in \Kc(L_2(\Circ^d_\theta)).
        \end{equation*}
        Finally then $\sum_{k=1}^n \tau_\theta(x_k)\pi_2(y_k) \in \Kc(L_2(\Circ^d_\theta))$,
        but this operator is in $\pi_2(C(\Sb^{d-1}))$ and so is compact if and only if it is zero.
        This contradicts the assumed linear independence of $\{y_k\}_{k=1}^n$.
    \end{proof}

\subsection{Injectivity for $\SU(2)$}
    The proofs for $\SU(2)$ are very similar to the proofs for $\Circ^d_\theta$, {\hightlight however instead of integrating over $\Circ^d$ with the action $z_t$ we integrate with respect
    to the right Haar measure. The following theorem is a simple and well known consequence of the invariance of the Haar measure on a compact group, and we include the proof for convenience.}
    \begin{lem}\label{sphere bochner}
        Let $x \in C(\SU(2))$. Consider the mapping
        \begin{equation*}
            z_g(x) := \lambda_r(g)\pi_1(x)\lambda_r(g)^* = \pi_1(\lambda_r(g)x).
        \end{equation*}
        The mapping $g\to z_g(x)$ is continuous in the norm topology of $\Bc(L_2(\SU(2)))$, and we have
        \begin{equation}\label{sphere integral}
            \int_{\SU(2)} z_g(x)\,dg = \int_{\SU(2)} x(g)\,dg 1
        \end{equation}
        where the integral on the left hand side is a Bochner integral in the operator norm, and the $1$
        on the right hand side is the identity operator.
    \end{lem}
    \begin{proof}
        Since $x$ is continuous on a compact space, it is uniformly continuous
        and so the mapping $g\to \lambda_r(g)x$ is continuous in the uniform norm. Hence, $g\to z_g$ is continuous
        in the operator norm.
        
        Due to the norm continuity of $\pi_1$,
        \begin{equation*}
            \pi_1\left(\int_{\SU(2)}\lambda_r(g)x\,dg\right) = \int_{\SU(2)} z_g(x)\,dg
        \end{equation*}
        By the right-invariance of the Haar measure, the left hand side is simply $$\int_{\SU(2)} x(g)\,dg \cdot \pi_1(1).$$
    \end{proof}
    
    {\hightlight  The following lemma, analogous to Lemma \ref{torus injectivity}, verifies Condition \eqref{symbc} of Theorem \ref{symbol def} for $C(\SU(2))$.}
    \begin{lem} \label{sphere injectivity}
        Let $x_k \in C(\SU(2))$ and $y_k \in \A_2$, $1 \leq k \leq n$. If 
        \begin{equation*}   
            \sum_{k=1}^n\pi_1(x_k)\pi_2(y_k) \in \Kc(L_2(\SU(2)))
        \end{equation*}
        then
        \begin{equation*}
            \sum_{k=1}^n x_k\otimes\rho_2(y_k) = 0.
        \end{equation*}
    \end{lem}
    \begin{proof}
        This proof is very similar to Lemma \ref{torus injectivity}. Again
        if $\sum_{k=1}^n x_k\otimes \rho_2(y_k) \neq 0$, then we may assume that the set $\{\rho_2(y_k)\}_{k=1}^n$
        is linearly independent to find a contradiction and we may also assume that there
        is at least one $k$ with $\int_{\SU(2)} x_k(g)\,dg \neq 0$. From Lemma \ref{sphere bochner}, 
        the mapping 
        \begin{equation*}
            g \mapsto \lambda_r(g)\left(\sum_{k=1}^n \pi_1(x_k)\pi_2(y_k)\right)\lambda_r(g)^*
        \end{equation*}
        is norm-continuous, hence 
        \begin{equation*}
            \int_{\SU(2)} \lambda_r(g)\left(\sum_{k=1}^n \pi_1(x_k)\pi_2(y_k)\right)\lambda_r(g)^* \,dg \in \Kc(L_2(\SU(2))).
        \end{equation*}
        However, $\lambda_r(g)$ commutes with each $D_j$ (since right actions commute with left actions), and hence with each $\pi_2(y_k)$. Thus from Lemma \ref{sphere bochner},
        \begin{equation*}
            \sum_{k=1}^n \int_{\SU(2)} x_k(g)\,dg\cdot \pi_2(y_k) \in \Kc(L_2(\SU(2))).
        \end{equation*}
        Applying the quotient map $q$, we obtain a linear dependence of $\{\rho_2(y_k)\}_{k=1}^d$, and this is a contradiction.
    \end{proof} 

\section{Verification of the commutator condition}
    We now verify the ``commuting modulo compacts" condition \ref{symbol def}\eqref{symbb} in each of our three examples: the noncommutative torus, noncommutative Euclidean space
    and $\SU(2)$.

    We cover only the case of noncommutative Euclidean space in detail, as the other examples are very similar. The following
    Lemma follows the proof of \cite[Proposition 2.14]{CGRS2}.
    \begin{lem}\label{ncp ver l1} If $x\in W^{2,2}(\mathbb{R}^d_{\theta}),$ then
    \begin{equation}\label{ncp ver l1 eq1}
    [(1-\Delta)^{\frac12},\pi_1(x)](1-\Delta)^{-\frac12}\in\mathcal{K}(L_2(\mathbb{R}^d)).
    \end{equation}
    \end{lem}
    \begin{proof}  We have
    $$(1-\Delta)^{-\frac12}=\frac1{\pi}\int_0^{\infty}\frac{d\lambda}{\lambda^{\frac12}(1+\lambda-\Delta)}.$$
    Thus,
    $$[(1-\Delta)^{\frac12},\pi_1(x)]=\frac1{\pi}\int_0^{\infty}\frac{d\lambda}{\lambda^{\frac12}}\Big[\frac{1-\Delta}{1+\lambda-\Delta},\pi_1(x)\Big].$$
    Denote for brevity
    $$A=(1-\Delta)^{-\frac12}[-\Delta,\pi_1(x)],\quad B=(1-\Delta)^{-1}[-\Delta,[-\Delta,\pi_1(x)]].$$
    A little algebra gives us
    $$\Big[\frac{1-\Delta}{1+\lambda-\Delta},\pi_1(x)\Big]=\Big(\frac{(1-\Delta)^{\frac12}}{(1+\lambda-\Delta)}-\frac{(1-\Delta)^{\frac32}}{(1+\lambda-\Delta)^2}\Big)A+\frac{\lambda(1-\Delta)}{(1+\lambda-\Delta)^2}B\frac1{1+\lambda-\Delta}.$$
    It is easy to see that
    $$\frac1{\pi}\int_0^{\infty}\Big(\frac{(1-\Delta)^{\frac12}}{(1+\lambda-\Delta)}-\frac{(1-\Delta)^{\frac32}}{(1+\lambda-\Delta)^2}\Big)\frac{d\lambda}{\lambda^{\frac12}}=\frac12.$$
    Thus, we can write (here, LHS denotes the left hand side in \eqref{ncp ver l1 eq1})
    $$LHS=\frac12 A(1-\Delta)^{-\frac12}+\frac1{\pi}\int_0^{\infty}\frac{d\lambda}{\lambda^{\frac12}}\frac{\lambda(1-\Delta)}{(1+\lambda-\Delta)^2}B\frac1{(1+\lambda-\Delta)(1-\Delta)^{\frac12}}.$$

    Clearly,
    $$A(1-\Delta)^{-\frac12}=\sum_{k=1}^d\frac{D_k}{(1-\Delta)^{\frac12}}\cdot [D_k,\pi_1(x)](1-\Delta)^{-\frac12}+\sum_{k=1}^d(1-\Delta)^{-\frac12}[D_k,\pi_1(x)]\cdot \frac{D_k}{(1-\Delta)^{\frac12}}.$$
    It follows from Theorem \ref{ncp cwikel lp} that
    $$[D_k,\pi_1(x)](1-\Delta)^{-\frac12},(1-\Delta)^{-\frac12}[D_k,\pi_1(x)]\in\mathcal{L}_{d+1}(L_2(\mathbb{R}^d)).$$
    Thus,
    $$A(1-\Delta)^{-\frac12}\in\mathcal{L}_{d+1}\mbox{ and, similarly, }B(1-\Delta)^{-\frac12}\in\mathcal{L}_{d+1}.$$

    Taking into account that
    $$\Big\|\frac{\lambda(1-\Delta)}{(1+\lambda-\Delta)^2}\Big\|_{\infty}\leq1,\quad\Big\|\frac1{1+\lambda-\Delta}\Big\|_{\infty}\leq\frac1{1+\lambda},$$
    we obtain
    $$\Big\|\frac{\lambda(1-\Delta)}{(1+\lambda-\Delta)^2}B\frac1{(1+\lambda-\Delta)(1-\Delta)^{\frac12}}\Big\|_{d+1}\leq\frac1{1+\lambda}\Big\|B(1-\Delta)^{-\frac12}\Big\|_{d+1}.$$
    Therefore,
    $$\|LHS\|_{d+1}\leq\frac12\|A(1-\Delta)^{-\frac12}\|_{d+1}+\frac1{\pi}\int_0^{\infty}\frac{d\lambda}{\lambda^{\frac12}(1+\lambda)}\cdot\|B(1-\Delta)^{-\frac12}\|_{d+1}.$$
    Since the right hand side is finite, the assertion follows.
    \end{proof}

    The operators $\frac{D_k}{\sqrt{-\Delta}}$, $k = 1,\ldots,d$ are the noncommutative equivalent of the Riesz transforms $R_k$. 
    The following Lemma can be viewed as a noncommutative variant of the classical result that if $f \in C_0(\Rl^d)$, then the commutators $[M_f,R_k]$
    are compact.
    \begin{lem}\label{ncp ver l2} If $x\in W^{2,2}(\mathbb{R}^d_{\theta}),$ then
    \begin{equation}\label{ncp ver l2 eq1}
    [\pi_1(x),\frac{D_k}{(-\Delta)^{\frac12}}]\in\mathcal{K}(L_2(\mathbb{R}^d)),\quad k = 1,\ldots,d
    \end{equation}
    \end{lem}
    \begin{proof} Firstly, we consider the commutator
    $$[\pi_1(x),\frac{D_k}{(1-\Delta)^{\frac12}}]=-[D_k,\pi_1(x)](1-\Delta)^{-\frac12}+\frac{D_k}{(1-\Delta)^{\frac12}}\cdot [(1-\Delta)^{\frac12},\pi_1(x)](1-\Delta)^{-\frac12}.$$
    Using Theorem \ref{ncp cwikel} for the first summand and Lemma \ref{ncp ver l1} for the second summand, we infer that
    $$[\pi_1(x),\frac{D_k}{(1-\Delta)^{\frac12}}]\in\mathcal{K}(L_2(\mathbb{R}^d)).$$

    Define a function $h_k$ on $\mathbb{R}^d$ by setting
    $$h_k(t)=\frac{t_k}{|t|}-\frac{t_k}{(1+|t|^2)^{\frac12}}=\frac{t_k}{|t|}\cdot\frac{1}{(1+|t|^2)^{\frac12}\cdot ((1+|t|^2)^{\frac12}+|t|)},\quad t\in\mathbb{R}^d.$$
    It follows from Theorem \ref{ncp cwikel} that
    $$[\pi_1(x),h_k(\nabla)]=\pi_1(x)h_k(\nabla)-h_k(\nabla)\pi_1(x)\in\mathcal{L}_{d+1}(\mathbb{R}^d_{\theta}).$$
    Thus,
    $$[\pi_1(x),\frac{D_k}{(-\Delta)^{\frac12}}]=[\pi_1(x),\frac{D_k}{(1-\Delta)^{\frac12}}]+[\pi_1(x),h_k(\nabla)]\in\mathcal{K}(L_2(\mathbb{R}^d)).$$
    \end{proof}

    Now we may complete the verifications of the condition \ref{symbol def}\eqref{symbb} for $\Rl^d_\theta$,
    \begin{thm}\label{ncp ver commutator} If $x\in C_0(\mathbb{R}^d_{\theta})$ and if $y\in C(\Sb^{d-1}),$ then $[\pi_1(x),\pi_2(y)]\in\mathcal{K}(L_2(\mathbb{R}^d)).$
    \end{thm}
    \begin{proof}
    Lemma \ref{ncp ver l2} shows that $[\pi_1(x),\pi_2(y)] \in \Kc(L_2(\Rl^d))$ when $x \in W^{2,2}(\Rl^d_\theta)$ and $y(t) = \frac{t_k}{|t|}$. 
    As $W^{2,2}(\Rl^d_\theta)$ is norm-dense in $C_0(\Rl^d_\theta)$ (see Lemma \ref{ncp cwikel}) and the compact operators
    are closed in the norm topology, the result follows for arbitrary $x \in C_0(\Rl^d_\theta)$ and $y(t) = \frac{t_k}{|t|}$.
    
    We may now extend the result to all $y$ given as a polynomial in the variables $\frac{t_k}{|t|}$ using the Leibniz rule. Finally
    by the Stone-Weierstrass theorem, we may approximate arbitrary $y \in C(\Sb^{d-1})$ by polynomials in the uniform norm. Hence
    again using the fact that $\Kc(L_2(\Rl^d))$ is norm-closed, this completes the proof.
    %
    \end{proof}

\section{Connes' Trace Formula in the examples}
    We now proceed to establish a variant of Connes' Trace Theorem which applies
    to our examples. 
    Let $H$ be a separable Hilbert space. We recall that a linear functional $\varphi:\Lc_{1,\infty}(H)\to \Cplx$ is called
    a continuous trace if $\varphi([A,B]) = 0$ for all $A \in \Bc(H)$ and $B \in \Lc_{1,\infty}(H)$
    and $|\varphi(B)| \lesssim \|B\|_{1,\infty}$. We will call $\varphi$ normalised if
    \begin{equation*}
        \varphi\left(\diag\left\{\frac{1}{n+1}\right\}_{n=0}^\infty\right) = 1.
    \end{equation*}

    If $\varphi$ is a normalised trace, then note also that
    \begin{equation*}
        \varphi\left(\diag\left\{\frac{1}{(1+|n|^2)^{d/2}}\right\}\right) = \frac{\mathrm{Vol}(\Sb^{d-1})}{d}.
    \end{equation*}
    It is known (see \cite[Corollary 5.7.7]{LSZ}) that any continuous trace $\varphi$ on $\Lc_{1,\infty}(H)$ vanishes on $\Kc(H)\cdot \Lc_{1,\infty}(H)$.

    Let $\A_1$, $\A_2$ be the algebras in any of our three examples. We establish that for any continuous normalised trace $\varphi$
    on $\Lc_{1,\infty}$, and $T \in \Pi(\A_1,\A_2)$,
    \begin{equation}\label{cif general}
        \varphi(T(1-\Delta)^{-d/2}) = c(\A_1,\A_2)\left(\tau_\theta\otimes \int_{\Sb^{d-1}}\right)(\sym(T)).
    \end{equation}
    where $\tau_\theta$ is replaced by integration with respect to the Haar measure in the example of $\SU(2)$, and $c(\A_1,\A_2)$ is a nonzero constant depending
    on the choices of $\A_1$ and $\A_2$.
    For \eqref{cif general} to be hold for noncommutative Euclidean space, we must make the additional assumption that there exists $z \in W^{d,1}(\Rl^d_\theta)$
    such that $T\pi_1(z) = T$. This result should be compared with \cite[Corollary 11.6.20]{LSZ}, which applies to classical
    pseudodifferential operators $T$ on $\Rl^d$ with the condition that there exists $\psi \in C^\infty_c(\Rl^d)$ such that $TM_{\psi} = T$. 
    
    {\hightlight  The formula \eqref{cif general} is our version of Connes' Trace Theorem, and we verify it in each of our three examples: first for noncommutative tori (Theorem \ref{ctt nc torus}), for $\SU(2)$ (Theorem \ref{ctt sphere}) and for noncommutative spaces
    (Theorem \ref{ctt nc plane}). }

    To establish \eqref{cif general} for all of our examples, we use the following two results. Lemma \ref{lifting lemma}
    follows immediately from the fact that any continuous normalised trace on $\Lc_{1,\infty}(H)$ vanishes on $\Kc(H) \cdot \Lc_{1,\infty}(H)$
    \begin{lem}\label{lifting lemma}
        Let $V \in \Lc_{1,\infty}(H)$, and let $\varphi$ be a continuous trace on $\Lc_{1,\infty}(H)$. Then,
        \begin{equation*}
            T \mapsto \varphi(TV)
        \end{equation*}
        is a continuous linear functional on $\Bc(H)$ which vanishes on $\Kc(H)$.
    \end{lem}
    
    {\hightlight To make our proof of Connes' trace theorem completely transparent, we return to the abstract setting of Theorem \ref{symbol def}. As with Theorem \ref{symbol def},
    the proof of the following lemma follows almost immediately from the assumptions.}
    \begin{lem}\label{tensor lemma}
        Let $\A_1,\A_2$ and $H$ be as in Theorem \ref{symbol def}. Suppose that $\omega$ is a continuous linear functional on $\Pi(\A_1,\A_2)$
        which vanishes $\Pi(\A_1,\A_2)\cap \Kc(H)$. Then there exists a unique linear functional $\rho$ on $\A_1\otimes_{\min} \A_2$ such that
        \begin{equation*}
            \omega(T) = \rho(\sym(T))
        \end{equation*}
        for all $T \in \Pi(\A_1,\A_2)$.
        
        If, in addition, we have $\psi_1 \in \A_1^*$ and $\psi_2 \in \A_2^*$, and
        \begin{equation*}
            \omega(\pi_1(a)\pi_2(b)) = \psi_1(a)\psi_2(b)
        \end{equation*}
        for all $a \in \A_1$ and $b \in \A_2$, then
        \begin{equation*}
            \omega(T) = (\psi_1\otimes\psi_2)(\sym(T))
        \end{equation*}
        or in other words, $\rho = \psi_1\otimes\psi_2$.
    \end{lem}
    \begin{proof}
        Since $\omega$ vanishes on $\Pi(\A_1,\A_2)\cap \Kc(H)$, $\omega$ descends to a linear functional $\tilde{\omega}$ on $\Pi(\A_1,\A_2)/(\Pi(\A_1,\A_2)\cap \Kc(H))$, which is simply $q(\Pi(\A_1,\A_2)$.
        Theorem \ref{symbol def} gives an isometric $*$-isomorphism $j: q(\Pi(\A_1,\A_2) \to \A_1\otimes_{\min} \A_2$. Defining $\rho = \tilde{\omega}\circ j^{-1}$
        gives the required linear functional.

        Now to prove that $\rho = \psi_1\otimes\psi_2$, first we note that it follows from Theorem \ref{linear functional tensor} that $\psi_1\otimes\psi_2$ is well defined
        on $\A_1\otimes_{\min} \A_2$. Since by assumption $\psi_1$ and $\psi_2$ are continuous,
        $\psi_1\otimes\psi_2$ is determined by its values on the algebraic tensor product $\A_1\otimes\A_2$. Hence, the linear functional $\psi_1\otimes\psi_2$ is uniquely characterised by
        \begin{equation*}
            (\psi_1\otimes\psi_2)(a\otimes b) = \psi_1(a)\psi_2(b)\,\quad a \in \A_1,\,b \in \A_2.
        \end{equation*}
        Since by assumption $\rho(a\otimes b) = \omega(\pi_1(a)\pi_2(b)) = \psi_1(a)\psi_2(b)$, it follows that $\rho = \psi_1\otimes\psi_2$.
    \end{proof}

    Lemmas \ref{lifting lemma} and \ref{tensor lemma} show that to establish \eqref{cif general} for the noncommutative torus
    it suffices to show that we have for all $a \in C(\Circ^d_\theta)$ and $g \in C(\Sb^{d-1})$,
    \begin{equation}\label{cif torus}
        \varphi(\pi_1(a)\pi_2(g)(1-\Delta)^{-d/2}) = c(\varphi,d)\tau_\theta(a)\int_{\Sb^{d-1}} g(t)\,dt,
    \end{equation}

    Recall that for $\SU(2)$, the algebra we denote $\A_2$ is generated by the operators $\{b_1,b_2,b_3\}$,
    and there is a map $\sym:\A_2\to C(\Sb^2)$. To establish \eqref{cif general}, we need to prove
    that for all $f \in C(\SU(2))$ and $g \in \A_2$,
    \begin{equation}\label{cif sphere}
        \varphi(\pi_1(f)\pi_2(g)(1-\Delta)^{-d/2}) = c(\varphi)\int_{\SU(2)}f(s)\,ds \int_{\Sb^{2}} \sym(g)(t)\,dt,
    \end{equation}

    The case for noncommutative Euclidean space is more subtle. Here, we fix $z \in W^{d,1}(\Rl^d_\theta)$,
    and consider the functional
    \begin{equation*}
        \omega(T) := \varphi(T\pi_1(z)(1-\Delta)^{-d/2})
    \end{equation*}
    and we must prove that for all $x \in C_0(\Rl^d_\theta)+\Cplx$ and $g \in C(\Sb^{d-1})$,
    \begin{equation}\label{cif plane}
        \omega(\pi_1(x)\pi_2(g)) = \tau_\theta(xz)\int_{\Sb^{d-1}} g(t)\,dt.
    \end{equation}
    From Lemma \ref{tensor lemma} it then follows
    that $\omega(T) = \left(\tau_\theta\otimes \int_{\Sb^{d-1}}\right)(\sym(T)(z\otimes 1))$.

    The remainder of this section contains the required argument for \eqref{cif torus}, \eqref{cif sphere} and \eqref{cif plane}.
    The statement and proof of Connes' Trace Formula for the commutative case $\Rl^d$ is given in \cite[Theorem 4]{Dao1}.

\subsection{Connes' Trace Formula on the Noncommutative Torus}
    In this subsection we prove \eqref{cif torus}. 

    \begin{lem}\label{first torus trace lemma} 
    If $y\in C(\Sb^{d-1}),$ then
    $$\varphi\left(\diag\left\{y\left(\frac{n}{|n|}\right)\cdot (1+|n|^2)^{-d/2}\right\}_{n \in \Itgr^d}\right)= \frac{1}{d}\int_{\Sb^{d-1}}y$$
    for every continuous normalised trace $\varphi$ on $\mathcal{L}_{1,\infty}.$ Here, $c_d = \frac{1}{d}$.
    \end{lem}
    \begin{proof} 
    We establish the equivalent result that,
    \begin{equation*}
        \sum_{|n| \leq N} y\left(\frac{n}{|n|}\right)(1+|n|^2)^{-d/2} = \int_{\Sb^{d-1}} y(t)\,dt\cdot  \log(N) + O(1)
    \end{equation*}
    as $N\to \infty$. The equivalence of this assertion to the result is established in \cite[Corollary 11.2.4]{LSZ}.

    First we suppose that $y$ is Lipschitz. In this case, we have
    \begin{equation*}
        y\left(\frac{n}{|n|}\right)(1+|n|^2)^{-d/2}-\int_{n+[0,1]^d} y\left(\frac{t}{|t|}\right)(1+|t|^2)^{-d/2}\,dt = O((1+|n|^2)^{-\frac{d+1}{2}}).
    \end{equation*}
    So summing over all $n$ with $|n| \leq N$,
    \begin{equation}\label{integral approx}
        \sum_{|n|\leq N} y\left(\frac{n}{|n|}\right)(1+|n|^2)^{-d/2} -\sum_{|n| \leq N}\int_{n+[0,1]^d}y\left(\frac{t}{|t|}\right)(1+|t|^2)^{-d/2}\,dt = O(1).
    \end{equation}
    The left hand side can be estimated as,
    \begin{align}\label{spherical approx}
        \left|\int_{|t| \leq N} y\left(\frac{t}{|t|}\right)(1+|t|^2)^{-d/2}\,dt - \sum_{|n|\leq N} \int_{n+[0,1]^d} y\left(\frac{t}{|t|}\right)(1+|t|^2)^{-d/2}\,dt\right| &\leq c_d \|y\|_\infty N^{d-1} N^{-d}\\
                                                                                                                                                            &= O(1).
    \end{align}
    So by \eqref{integral approx} and \eqref{spherical approx},
    \begin{equation*}
        \sum_{|n|\leq N} y\left(\frac{n}{|n|}\right)(1+|n|^2)^{-d/2} = \int_{|t|\leq N} y\left(\frac{t}{|t|}\right)(1+|t|^2)^{-d/2}\,dt + O(1).
    \end{equation*}
    The integral on the right hand side can be computed by a polar decomposition:
    \begin{equation*}
        \int_{|t|\leq N} y\left(\frac{t}{|t|}\right)(1+|t|^2)^{-d/2}) = \int_{\Sb^{d-1}} y(t)\,dt \int_{0}^N \frac{r^{d-1}}{(1+r^2)^{d/2}}\,dr.
    \end{equation*}
    However as, 
    \begin{equation*}
        \int_0^N \frac{r^{d-1}}{(1+r^2)^{d/2}}\,dr = \log(N) + O(1)
    \end{equation*}
    Hence there is a constant $c_d$ such that for all Lipschitz $y \in C(\Sb^{d-1})$,
    \begin{equation*}
        \varphi\left(\diag\left\{y\left(\frac{n}{|n|}\right)\cdot (1+|n|^2)^{-d/2}\right\}_{n \in \Itgr^d}\right)=c_d\int_{\Sb^{d-1}}y.
    \end{equation*}
    The value of $c_d$ can be determined by putting $y = 1$. 

    To remove the assumption that $y$ is Lipschitz, we note that both
    \begin{equation*}
        y \mapsto \int_{\Sb^{d-1}} y(t)\,dt
    \end{equation*} 
    and
    \begin{equation*}
        y \mapsto \varphi\left(\diag\left\{y\left(\frac{n}{|n|}\right)\cdot (1+|n|^2)^{-d/2}\right\}_{n \in \Itgr^d}\right)
    \end{equation*}
    are continuous in the uniform norm on $C(\Sb^{d-1})$, with the second case following from our assumption that $\varphi$ is continuous.

    Hence, as we may approximate an arbitrary $y \in C(\Sb^{d-1})$ in the uniform norm by a sequence $\{y_k\}_{k=0}^\infty$ of Lipschitz functions, the result follows.
    \end{proof}

    \begin{lem}\label{second torus trace lemma} If $x\in L_{\infty}(\mathbb{T}^d_{\theta})$ and if $y\in C(\Sb^{d-1}),$ then
    $$\varphi(\pi_1(x)\pi_2(y)(1-\Delta)^{-\frac{d}{2}})=\frac{1}{d}\tau_{\theta}(x)\cdot\int_{\Sb^{d-1}}y$$
    for every continuous normalised trace $\varphi$ on $\mathcal{L}_{1,\infty}.$
    \end{lem}
    \begin{proof} 
    We refer to \cite[Corollary 11.2.4]{LSZ}, which implies as a special case that if $V \in \Lc_{1,\infty}$ is nonnegative, with a
    sequence of eigenvectors $\{e_n\}_{n=0}^\infty$, and $T \in \Bc(H)$, then
    \begin{equation*}
        \varphi(TV) = \varphi\left(\diag\left\{\langle e_n,TVe_n\rangle \right\}_{n=0}^\infty \right)
    \end{equation*}
    We apply this result with $V = (1-\Delta)^{-d/2}$ and $T = \pi_1(x)\pi_2(y)$. Note that the unitary generators $\{u_n\}_{n\in \Itgr^d}$
    are eigenvectors for $(1-\Delta)^{-d/2}$. Now we obtain:
    \begin{equation*}
        \varphi(\pi_1(x)\pi_2(y)(1-\Delta)^{-\frac{d}{2}}) = \varphi\left(\diag\left\{\langle u_n,\pi_1(x)\pi_2(y)(1-\Delta)^{-\frac{d}{2}}u_n\rangle\right\}_{n\in\mathbb{Z}^d}\right).
    \end{equation*}
    Again by definition, $u_n,$ $n\in\mathbb{Z}^d,$ are eigenvectors for $\pi_2(y).$ Thus,
    \begin{align*}
        \langle u_n,\pi_1(x)\pi_2(y)(1-\Delta)^{-\frac{d}{2}}u_n\rangle &= \tau_\theta(u_n^*xy\left(\frac{n}{|n|}\right)(1+|n|^2)^{-d/2}u_n)\\
                                                                        &= \tau_{\theta}(x)\cdot y\left(\frac{n}{|n|}\right)\cdot (1+|n|^2)^{-\frac{d}{2}}.
    \end{align*}
    The assertion follows now from Lemma \ref{first torus trace lemma}.
    \end{proof}

    {\hightlight  By an application of Lemma \ref{tensor lemma}, we have the following theorem, which is our version of Connes' trace theorem for noncommutative tori.}
    \begin{thm}\label{ctt nc torus}
        If $T \in \Pi(C(\Circ^d_\theta),C(\Sb^{d-1}))$, then for any continuous normalised trace $\varphi$ on $\Lc_{1,\infty}$,
        \begin{equation*}
            \varphi(T(1-\Delta)^{-d/2}) = \frac{1}{d}\left(\tau_\theta\otimes\int_{\Sb^{d-1}}\right)(\sym(T)).
        \end{equation*}
    \end{thm}
    \begin{proof}
        From Lemma \ref{lifting lemma}, the functional
        \begin{equation*}
            \omega(T) = \varphi(T(1-\Delta)^{-d/2})
        \end{equation*}
        vanishes on $\Pi(C(\Circ^d_\theta),C(\Sb^{d-1}))\cap \Kc(L_2(\Circ^d_\theta))$. Due to Lemma \ref{second torus trace lemma},
        we can apply Lemma \ref{tensor lemma} to obtain $\omega(T) = \frac{1}{d}\left(\tau_\theta\otimes \int_{\Sb^{d-1}}\right)(\sym(T))$. 
    \end{proof}

\subsection{Connes' Trace Formula on $\SU(2)$}
    Recall that in the setting of $\SU(2)$, $\A_2$ denotes the $C^*$-algebra generated by $\frac{D_k}{\sqrt{-\Delta}}$, for $k = 1,2,3$. 

    Appendix \ref{identification of sphere algebra} shows that $q(\A_2)$ is in fact isometrically isomorphic to $C(\Sb^2)$, as
    in Appendix \ref{identification of sphere algebra}, we let $u:C(\Sb^2)\to q(\A_2)$ be the isomorphism.
    We have $(1\otimes u^{-1})\circ \sym:\mathcal{A}_2\to 1\otimes C(\Sb^2).$ Making
    a slight abuse of notation, we consider
    $\sym(x) \in C(\Sb^2)$ when $x \in \A_2$.

    \begin{lem}\label{first sphere trace lemma} 
    For every continuous trace $\varphi$ on $\mathcal{L}_{1,\infty},$ we have
    $$\varphi(x(1-\Delta)^{-\frac32})=\frac{\varphi((1-\Delta)^{-3/2})}{\Vol(\Sb^2)}\int_{\Sb^2}\sym(x)(t)\,dt,\quad x\in\mathcal{A}_2.$$
    \end{lem}
    \begin{proof}
        Since $\varphi$ is a trace on $\Lc_{1,\infty}$, it is vanishes on finite rank operators \cite[Corollary 5.7.7]{LSZ}. 
        By continuity, it follows that $\varphi$ vanishes on $\Kc\cdot \Lc_{1,\infty}$. Hence, $\rho(x(1-\Delta)^{-3/2})$
        depends only on the class of $x$ modulo compact operators. 
        
        From Appendix \ref{identification of sphere algebra}, there is an isometric $*$-isomorphism $u:C(\Sb^2)\to q(\A_2)$.
        Let $f \in C(\Sb^2)$, and choose $x \in \A_2$ such that $u(f) = q(x)$. Define a linear functional:
        \begin{equation*}
            L(f) := \varphi(x(1-\Delta)^{-d/2}).
        \end{equation*}
        The above is independent of the choice of $x$, since if $x_1,x_2$ are such that $q(x_1) = q(x_2)$, then $\varphi(x_1(1-\Delta)^{-3/2}) = \varphi(x_2(1-\Delta)^{-d/2})$.
        Note we also have that $L$ is continuous, since $|L(f)| \leq |\varphi((1-\Delta)^{-3/2})|\|q(x)\| = |\varphi((1-\Delta)^{-d/2})|\|u(f)\|$, and $\|u(f)\| = \|f\|_\infty$.
        
        Since the trace $\varphi$ is unitarily invariant,
        \begin{equation*}
            \varphi(x(1-\Delta)^{-3/2}) = \varphi(\lambda_l(g)x(1-\Delta)^{-3/2}\lambda_l(g)^{-1})
        \end{equation*}
        and since $\Delta$ commutes with $\lambda_l(g)$,
        \begin{equation*}
            \varphi(x(1-\Delta)^{-3/2}) = \varphi(\lambda_l(g)x\lambda_l(g)^{-1}(1-\Delta)^{-3/2}).
        \end{equation*}
        
        We introduce the canonical surjective map,
        \begin{equation*}
            \eta:\SU(2)\to \mathrm{SO}(3).
        \end{equation*}
        Lemma \ref{eta transformation} shows that for all $f \in C(\Sb^2)$, we have $\lambda_l(g)u(f)\lambda_l(g)^{-1} = u(f\circ \eta(g))$.
        Thus, $L(f) = L(f\circ \eta(g))$. Since $\eta:\SU(2)\to \mathrm{SO}(3)$ is surjective,
        this means that $L \in C(\Sb^2)^*$ is invariant under all rotations. Hence, $L$
        is the rotation invariant measure on $\Sb^2$.

    %
    \end{proof}

    \begin{lem}\label{second sphere trace lemma} For every continuous trace $\varphi$ on $\mathcal{L}_{1,\infty},$,
    $$\varphi(\pi_1(f)x(1-\Delta)^{-\frac32})=\frac{\varphi((1-\Delta)^{-3/2})}{\Vol(\Sb^2)}\int_{{\rm SU}(2)}f(g)\,dg\cdot \int_{\Sb^2}\sym(x),\quad f\in C({\rm SU}(2)),\ x\in\mathcal{A}_2.$$
    \end{lem}
    \begin{proof} Fix $x \in \A_2$. The mapping
    $$L:f\to \varphi(\pi_1(f)x(1-\Delta)^{-\frac32}),\quad f\in C({\rm SU}(2)),$$
    is a bounded linear functional on $C({\rm SU}(2)).$

    Let $g\in{\rm SU}(2).$ We have
    $$\lambda_r(g)\pi_1(f)\lambda_r(g)^{-1}=\pi_1(\lambda_r(g)f)$$,  and 
    since left actions commute with right actions, and $x$ is a function of the generators
    of the left action, we have $\lambda_r(g)x=x\lambda_r(g).$
    Thus,
    \begin{align*}
        L(f) &= \varphi(\lambda_r(g)\pi_1(f)x(1-\Delta)^{-\frac32}\lambda_r(g)^{-1})\\
            &= \varphi(\lambda_r(g)\pi_1(f)\lambda_r(g)^{-1}x(1-\Delta)^{-\frac32})\\
            &= \varphi(\pi_1(\lambda_r(g)f)x(1-\Delta)^{-\frac32})\\
            &= L(\lambda_r(g)f).
    \end{align*}

    By the uniqueness theorem for Haar measures, it follows that the only $\lambda_r-$invariant bounded linear functional on $C({\rm SU}(2))$ is an integral with respect to the Haar measure (up to a constant factor).
    Thus,
    $$\varphi(\pi_1(f)x(1-\Delta)^{-\frac32})=\int_{{\rm SU}(2)}f(g)\,dg\cdot F(x).$$
    Substituting $f=1,$ we obtain
    $$\varphi(\pi_1(f)x(1-\Delta)^{-\frac32})=\int_{{\rm SU(2)}}f(g)\,dg\cdot \varphi(x(1-\Delta)^{-\frac32}).$$
    The assertion follows now from Lemma \ref{first sphere trace lemma}.
    \end{proof}

    So by an identical argument to Theorem \ref{ctt nc torus},{\hightlight  which is our version of Connes' trace theorem for $\SU(2)$.}
    \begin{thm}\label{ctt sphere}
        For every continuous trace $\varphi$ on $\Lc_{1,\infty}$, and for all $T \in \Pi(C(\SU(2)),\A_2)$,
        \begin{equation*}
            \varphi(T(1-\Delta)^{-3/2}) = \frac{\varphi((1-\Delta)^{-3/2})}{\Vol(\Sb^2)}\left(\int_{\SU(2)}\otimes \int_{\Sb^2}\right)(\sym(T)).
        \end{equation*}
    \end{thm}

\subsection{Connes' Trace formula on noncommutative Euclidean space}

    The following assertion is proved in \cite{SZ_cif}, {\hightlight  see also the related result \cite[Proposition 4.17]{Gayral-Gracia-Bondia-Iochum-Schucker-Varily-moyal-planes-2004}.}
    \begin{thm}\label{cif ncplane} If $x\in W^{d,1}(\mathbb{R}^d_{\theta}),$ then $x(1-\Delta)^{-\frac{d}{2}}\in\mathcal{L}_{1,\infty}$ and
    there is a constant $C(d,\theta) > 0$ such that
    $$\varphi(x(1-\Delta)^{-\frac{d}{2}})=C(d,\theta)\tau_{\theta}(x)$$
    for every normalised continuous trace on $\mathcal{L}_{1,\infty}.$
    \end{thm}

    We also need a pair of important intermediate results from \cite{SZ_cif}. Firstly,
    \begin{lem}\label{main invariance lemma} If $F$ is a continuous functional on $W^{d,1}(\mathbb{R}^d_{\theta})$ such that
    $$F(x)=F(U(-t)xU(t)),\quad x\in W^{d,1}(\mathbb{R}^d_{\theta}),\quad t\in\mathbb{R}^d,$$
    then $F=\tau_{\theta}$ (up to a constant factor).
    \end{lem}

    Let $M_d(\Rl)$ be the space of $d\times d$ real matrices. We define
    \begin{equation*}
        {\rm Sp}(\theta,d) := \Big\{g\in M_d(\mathbb{R}):\ g^*\theta g=\theta\Big\}.
    \end{equation*}
    As we are working under the assumption that $\det(\theta) \neq 0$, it follows that $\mathrm{Sp}(\theta,d)$
    is a group under usual matrix multiplication.

    By our assumption that $\det(\theta) \neq 0$, it follows that if $g \in \Sp(\theta,d)$ then $|\det(g)| = 1$.

    The second result from \cite{SZ_cif} we require is
    \begin{lem}\label{SL_d invariance of the trace}
        Let $g \in \Sp(\theta,d)$. We define an action $g\mapsto W_g$ on $L_2(\Rl^d)$ by
        \begin{equation*}
            (W_g\xi) = \xi\circ g^{-1}.
        \end{equation*}
        The operator $W_g$ is unitary on $L_2(\Rl^d)$, and conjugation by $W_g$ defines a trace-preserving group of automorphisms of $L_\infty(\Rl^d_\theta)$.
    \end{lem}
    Note that the assumption that $g \in \Sp(\theta,d)$ in Lemma \ref{SL_d invariance of the trace} is crucial: otherwise
    we do not necessarily have that $W_gxW_g^* \in L_\infty(\Rl^d_\theta)$ when $x \in L_\infty(\Rl^d_\theta)$.
    Let $\Omega$ be the antisymmetric matrix $\Omega := \begin{pmatrix} 0 & -1\\1 & 0 \end{pmatrix}^{\oplus d/2}$.
    Then $\Sp(\Omega,d)$ is the usual symplectic group.

    Let $g \in \mathrm{GL}(d,\Rl)$. Referring to Appendix A, consider the operator $V_g$ on $C(\Sb^{d-1})$ defined by
    \begin{equation*}
        (V_gf)(t) = \frac{1}{|gt|^d}f\left(\frac{gt}{|gt|}\right).
    \end{equation*}
    It can be easily verified that $g \mapsto V_g$ is an ``opposite group action" in the sense that it satisfies the rule $V_{gh} = V_h\circ V_g$ for all $g,h \in \mathrm{GL}(d,\Rl)$. 
    Lemma \ref{invariance lemma} proves that the rotation-invariant integration functional $m$ on $C(\Sb^{d-1})$ transforms
    under $V_g$ by $m\circ V_g = (\det(g))^{-1}m$.

    \begin{lem}\label{so teta} Let $l\in C(\Sb^{d-1})^*.$ If $l\circ V_g=l$ for every $g\in {\rm Sp}(\theta,d),$ then $l= \alpha m$ for 
        some $\alpha \in \Cplx$.
    \end{lem}
    \begin{proof} 

    The following result of linear algebra is well known, and follows easily from \cite[Section 9.44]{shilov}. 
    There exists a real invertible matrix $\beta$ with $\beta\beta^* = \beta^*\beta = |\det(\theta)|^{-1}$ such that
    \begin{equation}\label{almost diagonalisation}
        \beta^* \theta \beta = \Omega.
    \end{equation}
    Hence, if $g \in \Sp(\Omega,d)$ is arbitrary, then:
    \begin{align*}
        (\beta g\beta^{-1})^*\theta (\beta g \beta^{-1}) &= (\beta^*)^{-1}g^*\beta^*\theta \beta g \beta^{-1}\\
                                                        &= (\beta^*)^{-1} \Omega \beta^{-1}\\
                                                        &= \theta,
    \end{align*}
    so $\beta g \beta^{-1} \in \Sp(\theta,d)$.
    Since by assumption, $l\circ V_h = l$ for all $h \in \Sp(\theta,d)$, we have:
    \begin{equation*}
        l\circ V_\beta^{-1}\circ V_g\circ V_{\beta} = l.
    \end{equation*}
    Therefore for arbitrary $g \in \Sp(\theta,d)$,
    \begin{equation*}
        (l\circ V_{\beta^{-1}})\circ V_g = l\circ V_{\beta^{-1}},\quad \text{ for all }g \in \Sp(\Omega,d).
    \end{equation*}

    So by Theorem \ref{main symplectic thm}, there is a constant $C$ such that $l\circ V_{\beta^{-1}} = Cm$.
    Hence $l = Cm\circ V_{\beta}$. By Lemma \ref{invariance lemma}, $m\circ V_{\beta} = \det(\beta)^{-1}m$. Let $\alpha = C\det(\beta)^{-1}$,
    so that $l = \alpha m$.
    \end{proof}

    \begin{lem}\label{ncplane intermediate lemma}
        Let $\varphi$ be a continuous trace on $\Lc_{1,\infty}$.
        There is a continuous functional $l \in C(\Sb^{d-1})^*$ such that for all $x \in W^{d,1}(\Rl^d_\theta)$
        and all $b \in C(\Sb^{d-1})$ we have
        \begin{equation*}
                \varphi(\pi_1(x)\pi_2(b)(1-\Delta)^{-\frac{d}{2}})=\tau_{\theta}(x)\cdot l(b).
        \end{equation*}
    \end{lem}
    \begin{proof}
    Since $\varphi$ is unitarily invariant, it follows that
    $$\varphi(\pi_1(x)\pi_2(b)(1-\Delta)^{-\frac{d}{2}})=\varphi(e^{i\langle\theta t,\nabla\rangle}\pi_1(x)\pi_2(b)(1-\Delta)^{-\frac{d}{2}}e^{-i\langle\theta t,\nabla\rangle})$$
    However, $\nabla$ commutes with $\Delta$ and with $\pi_2(b).$ Thus,
    $$\varphi(\pi_1(x)\pi_2(b)(1-\Delta)^{-\frac{d}{2}})=\varphi(e^{i\langle\theta t,\nabla\rangle}\pi_1(x)e^{-i\langle\theta t,\nabla\rangle}\pi_2(b)(1-\Delta)^{-\frac{d}{2}}).$$
    
    Note that if $\xi \in L_2(\Rl^d)$,
    \begin{align*}
        e^{i\langle\theta t,\nabla\rangle}U(s)e^{-i\langle\theta t,\nabla\rangle}\xi(r) &= e^{i\langle \theta t,\nabla\rangle}U(s)e^{-i(\theta t,r)}\xi(r)\\
                                                                                        &= e^{i\langle \theta t,\nabla\rangle}e^{\frac{i}{2}(s,\theta r)-i(\theta t,r-s)}\xi(r-s)\\
                                                                                        &= e^{i(\theta t,r)+\frac{i}{2}(s,\theta r)-i(\theta t,r)+i(\theta t,s)}\xi(r-s)\\
                                                                                        &= e^{i(\theta t,s)}(U(s)\xi)(r)
    \end{align*}
    On the other hand from \eqref{canonical commutation relations},
    $$U(-t)U(s)U(t)=e^{i\langle \theta t,s\rangle}U(s).$$
    Since the family $\{U(t)\}_{t \in \Rl^d}$ generates $L_\infty(\Rl^d_\theta)$, it follows that
    for all $x \in L_\infty(\Rl^d_\theta)$ we have:
    $$e^{i\langle\theta t,\nabla\rangle}xe^{-i\langle\theta t,\nabla\rangle}=U(-t)xU(t).$$
    Since $\pi_1$ is actually the identity function, this is equivalent to
    \begin{equation*}
        e^{i\langle \theta t,\nabla\rangle}\pi_1(x)e^{-i\langle \theta t,\nabla\rangle} = \pi_1(U(-t)xU(t)).
    \end{equation*}
    
    Hence,
    \begin{equation*}
        \varphi(\pi_1(x)\pi_2(b)(1-\Delta)^{-\frac{d}{2}}) = \varphi(\pi_1(U(-t)xU(t))\pi_2(b)(1-\Delta)^{-d/2}).
    \end{equation*}    
    
    Consider now the linear functional on $W^{d,1}(\Rl^d_\theta)$,
    \begin{equation*}
        F(x) = \varphi(\pi_1(x)\pi_2(b)(1-\Delta)^{-d/2})
    \end{equation*}
    From Corollary \ref{ncp specific cwikel}, $F$ is continuous in the $W^{d,1}$-norm. We
    have proved that $F(U(-t)xU(t)) = F(x)$, and so from Lemma \ref{main invariance lemma}
    we can conclude that $F(x)$ is a scalar multiple of $\tau_\theta(x)$. 
    So,
    \begin{equation}\label{Cross-term}
        \varphi(\pi_1(x)\pi_2(b)(1-\Delta)^{-\frac{d}{2}})=\tau_{\theta}(x)\cdot l(b),
    \end{equation}
    for some functional $l$ on $C(\Sb^{d-1})$. Since $\varphi$ is continuous,
    \begin{equation*}  
        |l(b)| \leq C\|b\|_{\infty}
    \end{equation*} 
    for some $C \geq 0$. So $l$ is continuous.
    \end{proof}

    \begin{lem}\label{ncplane summary lemma} 
    Let $x \in W^{d,1}(\Rl^d_\theta)$ and $b \in C(\Sb^{d-1})$, then for any continuous normalised
    trace $\varphi$ on $\Lc_{1,\infty}$.
    $$\varphi(\pi_1(x)\pi_2(b)(1-\Delta)^{-\frac{d}{2}})=\frac{C(d,\theta)}{\Vol(\Sb^{d-1})}\tau_\theta(x)\int_{\Sb^{d-1}} b(t)\,dt.$$
    where $C(d,\theta)$ is the same constant as in Theorem \ref{cif ncplane}.
    \end{lem}
    \begin{proof} 
    Let $l$ be the linear functional from Lemma \ref{ncplane intermediate lemma}. It is required
    to show that we have:
    \begin{equation*}
        l(b) = \frac{C(d,\theta)}{\Vol(\Sb^{d-1})}\int_{\Sb^{d-1}} b(t)\,dt.
    \end{equation*}
    From Lemma \ref{so teta}, it suffices to show that $l\circ V_g = l$ for all $g \in \Sp(d,\theta)$, and we will
    be able to recover the constant by substituting $b=1$.
    
    Now let $g \in \Sp(\theta,d)$. Since the operator $W_g$ from Lemma \ref{SL_d invariance of the trace}
    is unitary, it follows that:
    \begin{align}
        \tau_\theta(x)l(b) &= \varphi(W_g^*\pi_1(x)\pi_2(b)(1-\Delta)^{-d/2}W_g)\\
                        &= \varphi(\pi_1(W_g^* x W_g)W_g^* \pi_2(b) (1-\Delta)^{-d/2}W_g).\label{tau l formula}
    \end{align}
    
    We now show that for all $y \in W^{d,1}(\Rl^d_\theta)$,
    \begin{equation}\label{Fourier side symplectic action}
        \pi_1(y)W_g^*\pi_2(b)(1-\Delta)^{-d/2}W_g - \pi_1(y)\pi_2(V_g b)(1-\Delta)^{-d/2} \in \Lc_1.
    \end{equation}
    
        Let $\xi \in L_2(\Rl^d)$, then 
        \begin{align*}
            (W_g^*\pi_2(b)(1-\Delta)^{-d/2} W_g\xi)(t) &= W_g^*(b\left(\frac{t}{|t|}\right)(1+|t|^2)^{-d/2}\xi(g^{-1}t))\\
                                                        &= b\left(\frac{gt}{|gt|}\right)(1+|gt|^2)^{-d/2}\xi(t)\\
                                                        &= b\left(\frac{gt}{|gt|}\right)\frac{|t|^d}{|gt|^d}\frac{|gt|^d}{|t|^d}(1+|gt|^2)^{-d/2}\xi(t).
        \end{align*}
        The above computation shows that:
        \begin{equation*}
            W_g^*\pi_2(b)(1-\Delta)^{-d/2} W_g = \pi_2(V_gb)\frac{|g\nabla|^d}{|\nabla|^d}(1+|g\nabla|^2)^{-d/2}.
        \end{equation*}
        Hence,
        \begin{align*}
        \pi_1(y)W_g^*\pi_2(b)&(1-\Delta)^{-d/2}W_g-\pi_1(y)\pi_2(V_g b)(1-\Delta)^{-d/2} \\
                            &= \pi_1(y)\pi_2(V_g(b))\left(\frac{|g\nabla|^d}{|\nabla|^d}(1+|g\nabla|^2)^{-d/2}-(1-\Delta)^{-d/2}\right)\\
                            &= \pi_1(y)\left(\frac{|g\nabla|^d}{|\nabla|^d}(1+|g\nabla|^2)^{-d/2}-(1-\Delta)^{-d/2}\right)\pi_2(V_g(b)).
        \end{align*}
        Due to Lemma \ref{ncp trace class cwikel}, to prove \eqref{Fourier side symplectic action}, it suffices to show that:
        \begin{equation*}
            h(t) := \frac{|gt|^d}{|t|^d}(1+|gt|^2)^{-d/2}-(1+|t|^2)^{-d/2}.
        \end{equation*}
        is in $\ell_1(L_\infty(\Rl^d))$. It is clear that $h$ is bounded in the ball $\{|t|\leq 1\}$.
        Supposing $|t| > 1$, we rewrite $h$ as,
        \begin{equation*}
            h(t) = |t|^{-d}\left(\frac{|gt|^d}{(1+|gt|^2)^{d/2}}-\frac{|t|^d}{(1+|t|^2)^{d/2}}\right)
        \end{equation*} 
        Since $\frac{|gt|^2}{1+|gt|^2}$ and $\frac{|t|^2}{1+|t|^2}$ are bounded above by $1$, we may use the numerical inequality:
        \begin{equation*}
            |\alpha^{d/2}-\beta^{d/2}| \leq \frac{d}{2}|\alpha-\beta|,\quad |\alpha|,|\beta| \leq 1
        \end{equation*}
        to obtain,
        \begin{equation*}
            |h(t)| \leq \frac{d}{2}|t|^{-d}\left|\frac{|gt|^2}{1+|gt|^2}-\frac{|t|^2}{1+|t|^2}\right|
        \end{equation*}
        However,
        \begin{align*}
            \frac{|gt|^2}{1+|gt|^2}-\frac{|t|^2}{1+|t|^2} &= (1+|t|^2)^{-1}\frac{|gt|^2-|t|^2}{1+|gt|^2}\\
                                                        &= O((1+|t|^2)^{-1}), \quad |t|\to\infty.
        \end{align*}
        Hence, $|h(t)| = O(|t|^{-d-2})$ as $|t|\to\infty$. From there it is easy to see that $h \in \ell_1(L_\infty)(\Rl^d)$.
        This completes the proof of \eqref{Fourier side symplectic action}.
    
        As $\varphi$ vanishes on $\Lc_1$, we may use \eqref{Fourier side symplectic action} with $y = W_g^* x W_g$
        to obtain in \eqref{tau l formula} to obtain,
        \begin{equation*}
            \tau_\theta(x)l(b) = \varphi(\pi_1(W_g^*xW_g)\pi_2(V_gb)(1-\Delta)^{-d/2})
        \end{equation*} 
        so by Lemma \ref{ncplane intermediate lemma}:
        \begin{equation*}
            \tau_\theta(x)l(b) = \tau_\theta(W_g^*xW_g)l(V_g b).
        \end{equation*} 
        From Lemma \ref{SL_d invariance of the trace} we have $\tau_\theta(W_g^*xW_g) =\tau_\theta(x)$,
        so now
        \begin{equation*}
            \tau_\theta(x)l(b) = \tau_\theta(x)l(V_g b).
        \end{equation*}
        Since $x \in W^{d,1}(\Rl^d_\theta)$ is arbitrary, it follows that $l(b) = l(V_g b)$. 
        So from Lemma \ref{so teta}, $l(b) = \alpha\int_{\Sb^{d-1}} b(t)\,dt$ for some constant $\alpha$. 
        By substituting $b = 1$ and using Theorem \ref{cif ncplane}, we recover the constant $\alpha$.
    %
    %
    \end{proof}

    {\hightlight  Finally, we have our version of Connes' trace theorem for $\Rl^d_\theta$:}
    \begin{thm}\label{ctt nc plane}
    Let $z \in W^{d,1}(\Rl^d_\theta)$. Then for every continuous
    normalised trace $\varphi$ on $\Lc_{1,\infty}$, and every $T \in \Pi(C_0(\Rl^d_\theta)+\Cplx,C(\Sb^{d-1}))$,
    \begin{equation*}
        \varphi(T\pi_1(z)(1-\Delta)^{-d/2}) = \frac{C(d,\theta)}{\Vol(\Sb^{d-1})}\left(\tau_\theta\otimes \int_{\Sb^{d-1}}\right)(\sym(T)(z\otimes 1)).
    \end{equation*}
    In particular, if $T = T\pi_1(z)$, then
    \begin{equation*}
        \varphi(T(1-\Delta)^{-d/2}) = \frac{C(d,\theta)}{\Vol(\Sb^{d-1})}\left(\tau_\theta\otimes \int_{\Sb^{d-1}}\right)(\sym(T)).
    \end{equation*}
    Once again, $C(d,\theta)$ is the same constant as in Theorem \ref{cif ncplane}.
    \end{thm}
    \begin{proof} 
        We apply Lemma \ref{tensor lemma} to the functional
        \begin{equation*}
            \omega(T) = \varphi(T\pi_1(z)(1-\Delta)^{-d/2}).
        \end{equation*}
        Since $\pi_1(z)(1-\Delta)^{-d/2} \in \Lc_{1,\infty}$, it follows from Lemma \ref{ncp cwikel} that this functional is well defined
        and vanishes on compact operators. Consider the functionals $\psi_1(x) := C(d,\theta)\tau_\theta(xz)$ and $\psi_2(b) = \frac{1}{\Vol(\Sb^{d-1})}\int_{\Sb^{d-1}} b(t)\,dt$ on $\Cplx+C_0(\Rl^d_\theta)$
        and $C(\Sb^{d-1})$ respectively. From Lemma \ref{tensor lemma}, to show that $\omega(T) = (\psi_1\otimes\psi_2)(\sym(T))$ it suffices to prove:
        \begin{equation*}
            \omega(\pi_1(x)\pi_2(b)) = \psi_1(x)\psi_2(b).
        \end{equation*}
        To this end, we compute $\omega(\pi_1(x)\pi_2(b))$. Since $[\pi_1(x),\pi_2(b)]$
        is compact,
        \begin{equation*}
            \omega(\pi_1(x)\pi_2(b)) = \omega(\pi_2(b)\pi_1(x)).
        \end{equation*}
        Hence
        \begin{align*}
            \omega(\pi_1(x)\pi_2(b)) &= \varphi(\pi_2(b)\pi_1(x)\pi_1(z)(1-\Delta)^{-d/2})\\
                                    &= \varphi(\pi_2(b)\pi_1(xz)(1-\Delta)^{-d/2}).
        \end{align*}
        Using the cyclicity of the trace $\varphi$, and that $\pi_2(b)$ commutes with $\Delta$,
        \begin{equation*}
            \omega(\pi_1(x)\pi_2(b)) = \varphi(\pi_1(xz)\pi_2(b)(1-\Delta)^{-d/2}).
        \end{equation*}
        The right hand side may be computed using Lemma \ref{ncplane summary lemma},
        \begin{align*}
            \varphi(\pi_1(xz)\pi_2(b)(1-\Delta)^{-d/2}) &= \frac{C(d,\theta)}{\Vol(\Sb^{d-1})}\tau_\theta(xz)\int_{\Sb^{d-1}} b(t)\,dt\\
                                                        &= \psi_1(x)\psi_2(b).
        \end{align*}
        So finally, we have $\omega(\pi_1(x)\pi_2(b)) = \psi_1(x)\psi_2(b)$. So from Lemma \ref{tensor lemma}, we immediately
        obtain $\omega = \psi_1\otimes\psi_2$, and this completes the proof.
    %
    \end{proof}
    
{\hightlight 
\section{Conclusion}
    Having verified the conditions of Theorem \ref{symbol def} for each of our three examples, we have in each setting an algebra of order $0$ pseudodifferential operators and an operator-norm continuous principal symbol map. 
    It would be of interest to extend our methods to further examples, for example Lie groups more general than $\SU(2)$. A further generalisation would be to consider the principal symbol
    mapping of Theorem \ref{symbol def} when both algebras $\A_1$ and $\A_2$ are noncommutative. In fact, the only reason we have restricted attention to the case where $\A_2$ is commutative is due to the fact
    that commutative algebras are nuclear. Theorem \ref{symbol def} would work without modification in the case that both algebras are noncommutative, but at least one is nuclear.
    
    In general, in the setting of a compact $d$-dimensional Riemannian manifold $X$, the principal symbol of a pseudodifferential operator is a function on the cosphere bundle $S^*X$ (see the discussion under Definition 4.1 on page 36 of \cite{Shubin}). The symbol mapping
    in Theorem \ref{symbol def} takes values in the tensor product $\A_1\otimes_{\min}\A_2$. In geometric terms, when $\A_1 = C(X)$ and $\A_2 = C(Y)$ we have that $q(\Pi(\A_1,\A_2))$ is isomorphic to $C(X\times Y)$. Hence
    requiring that $q(\Pi(\A_1,\A_2))$ be a tensor product restricts attention to the case where the cosphere bundle
    is trivial, that is, when $S^*X$ is homeomorphic to $\Sb^{d-1}\times X$. A future extension of this work would need to go beyond the case where $q(\Pi(\A_1,\A_2))$ is isomorphic to a tensor product
    of $\A_1$ and $\A_2$.
}

\appendix

\section{Measures invariant under the action of symplectic groups}

    For $g \in \mathrm{GL}(d,\Rl)$, we define the following action $V_g$ on $C(\Sb^{d-1})$ as follows:
    \begin{equation*}
        (V_gb)(t) = \frac{1}{|gt|^d}b\left(\frac{gt}{|gt|}\right),\quad t \in \Sb^{d-1}.
    \end{equation*}
    It is indeed an (opposite) action: we have
    $$V_{g_1}\circ V_{g_2}=V_{g_2g_1},\quad g_1,g_2\in \mathrm{GL}(d,\Rl).$$

\begin{lem}\label{invariance lemma}
If $m$ is a rotation-invariant integration functional on $C(\Sb^{d-1}),$ then $m\circ V_g={\rm det}(g^{-1})\cdot m.$
\end{lem}
\begin{proof} 
    By converting to polar coordinates, for every $b \in C(\Sb^{d-1})$ we have the formula,
    \begin{equation*}
        m(b) = \frac{1}{\Gamma(d)}\int_{\Rl^d} b\left(\frac{t}{|t|}\right)e^{-|t|}\,dt
    \end{equation*}
    So,
    \begin{equation*}
        m(V_g b) = \frac{1}{\Gamma(d)} \int_{\Rl^d} b\left(\frac{gt}{|gt|}\right)\frac{|t|^d}{|gt|^d}e^{-|t|}\,dt.
    \end{equation*}
    Applying the linear transformation $s = gt$, we get,
    \begin{align*}
        m(V_g b) &= \frac{1}{\Gamma(d)} \int_{\Rl^d} b\left(\frac{s}{|s|}\right)\frac{|g^{-1}s|^d}{|s|^d}e^{-|g^{-1}s|}\,d(g^{-1}s)\\
                 &= \frac{\det(g^{-1})}{\Gamma(d)} \int_{\Rl^d} b\left(\frac{s}{|s|}\right)\frac{|g^{-1}s|^d}{|s|^d}e^{-|g^{-1}s|}\,ds.
    \end{align*}
    Now using polar coordinates,
    \begin{align*}
        m(V_g b) &= \frac{\det(g)^{-1}}{\Gamma(d)} \int_{\Sb^{d-1}} b(s) \int_0^\infty |g^{-1}s|^de^{-r|g^{-1}s|}r^{d-1}\,drds.
    \end{align*}
    Now using the formula $\Gamma(d) = \alpha^d\int_0^\infty r^{d-1}e^{-\alpha r}\,dr$, we get
    \begin{align*}
        m(V_g b) &= \frac{\det(g^{-1})}{\Gamma(d)} \Gamma(d)\int_{\Sb^{d-1}} b(s)\,ds\\
                 &= \det(g^{-1})m(b).
    \end{align*}
\end{proof}

Let $d$ be even. The symplectic group ${\rm Sp}(d,\mathbb{R})$ (a subgroup in $\mathrm{GL}(d,\Rl)$) is defined as follows 
$${\rm Sp}(d,\mathbb{R})=\Big\{g\in M_d(\mathbb{R}):\ g^*\Omega g=\Omega\Big\},\quad \Omega=
\begin{pmatrix}
0&1\\
-1&0
\end{pmatrix}^{\oplus \frac{d}{2}}.
$$

The remainder of this section is devoted to the proof of the following:

\begin{thm}\label{main symplectic thm} If $l\in C(\Sb^{d-1})^*$ is such that $l\circ V_g=l,$ for all $g\in {\rm Sp}(d,\mathbb{R}),$ then $l={\rm const}\cdot m.$
\end{thm}

Let $\mathrm{Poly}(\Sb^{d-1})$ denote the set of polynomials in the variables $t_1,\ldots,t_d$ on $\Sb^{d-1}$. Let ${\bf n}=(n_1,\cdots,n_d)\in\mathbb{Z}_+^d.$ Define,
\begin{equation*}
b_{{\bf n}}(t)=\prod_{k=1}^dt_k^{n_k},\quad t\in\mathbb{S}^{d-1},
\end{equation*}
considered as an element of $\mathrm{Poly}(\Sb^{d-1}).$ 

Our first result reduces the problem to even elements of $\mathbb{Z}_+^d:$

\begin{lem}\label{odd lemma} Let $l\in C(\Sb^{d-1})^*$ be such that $l\circ V_g=l,$ for all $g\in {\rm Sp}(d,\mathbb{R}).$ If ${\bf n}\in\mathbb{Z}_+^d$ is such that there is at least one $j$ with $n_j$ odd, then $l(b_{{\bf n}})=0.$
\end{lem}
\begin{proof} Since ${\rm SU}(2)$ is abelian, it follows that ${\rm SU}(2)\times\cdots\times{\rm SU}(2)\subset {\rm Sp}(d,\mathbb{R})$ (here are $\frac{d}{2}$ factors). Clearly, $|gt|=1$ for every $g\in {\rm SU}(2)\times\cdots\times{\rm SU}(2)$ and for every $t\in\mathbb{S}^{d-1}.$ Thus, $V_gb=b\circ g$ for every $g\in {\rm SU}(2)\times\cdots\times{\rm SU}(2).$ By assumption, we have
$$l(b)=l(V_gb)=l(b\circ g),\quad g\in{\rm SU}(2)\times\cdots\times{\rm SU}(2)\subset {\rm Sp}(d,\mathbb{R}).$$
Since $l$ is continuous on $C(\mathbb{S}^{d-1})$ and since the integral below is Bochner, it follows that
$$l(b)=l\Big(\int_{{\rm SU}(2)\times\cdots\times{\rm SU}(2)}(b\circ g)dg\Big),\quad b\in C(\mathbb{S}^{d-1}).$$
Here, $dg$ is the normalised Haar measure on the group ${\rm SU}(2)\times\cdots\times{\rm SU}(2).$

For every ${\bf n}\in\mathbb{Z}_+^d$ and for every $1\leq k\leq\frac{d}{2},$ we write
$${\bf n}^k=(0,\cdots,0,n_{2k-1},n_{2k},0,\cdots,0)\in\mathbb{Z}^d_+.$$
Every element in ${\rm SU}(2)\times\cdots\times{\rm SU}(2)$ can be written as
$$g_s=\prod_{k=1}^{\frac{d}{2}}g_{s_k},\quad s=(s_1,\cdots, s_{\frac{d}{2}})\in(0,2\pi)^{\frac{d}{2}},$$
where
$$g_{s_k}=1_{{\rm SU}(2)}\times\cdots\times 1_{{\rm SU}(2)}\times
\begin{pmatrix}
\cos(s_k)&\sin(s_k)\\
-\sin(s_k)&\cos(s_k)
\end{pmatrix}\times 1_{{\rm SU}(2)}\times\cdots\times 1_{{\rm SU}(2)}.$$
Clearly,
$$b_{{\bf n}}\circ g_s=\prod_{k=1}^{\frac{d}{2}}b_{{\bf n}^k}\circ g_{s_k}.$$
By the Fubini Theorem, we have
$$\int_{{\rm SU}(2)\times\cdots\times{\rm SU}(2)}(b_{{\bf n}}\circ g)dg=(2\pi)^{-\frac{d}{2}}\prod_{k=1}^{\frac{d}{2}}\int_{(0,2\pi)}b_{{\bf n}^k}\circ g_{s_k}ds_k.$$

For every $t\in\mathbb{S}^{d-1},$ we use the polar notation $t_{2k-1}=r_k\cos(\phi_k)$ and $t_{2k}=r_k\sin(\phi_k).$ We have
$$(b_{{\bf n}^k}\circ g_{s_k})(t)=r_k^{n_{2k-1}+n_{2k}}\cdot \cos^{n_{2k-1}}(\phi_k-s_k)\sin^{n_{2k}}(\phi_k-s_k).$$
Thus,
$$\int_{(0,2\pi)}(b_{{\bf n}^k}\circ g_{s_k})(t)ds_k=r_k^{n_{2k-1}+n_{2k}}\cdot\int_0^{2\pi}\cos^{n_{2k-1}}(\phi_k-s_k)\sin^{n_{2k}}(\phi_k-s_k)ds_k=$$
$$=r_k^{n_{2k-1}+n_{2k}}\cdot\int_0^{2\pi}\cos^{n_{2k-1}}(s_k)\sin^{n_{2k}}(s_k)ds_k.$$
If $n_{2k}$ is odd, then the substitution $s_k\to -s_k$ changes the sign of the latter integral, which, therefore, vanishes. If $n_{2k-1}$ is odd, then the substitution $s_k\to \pi-s_k$ changes the sign of the latter integral, which, therefore, vanishes. Since either $n_{2k-1}$ or $n_{2k}$ is odd for some $1\leq k\leq\frac{d}{2},$ it follows that at least one of these integrals vanishes. This completes the proof.
\end{proof}

Recall that the Lie algebra $\mathfrak{sp}(d,\mathbb{R})$ of ${\rm Sp}(d,\mathbb{R})$ is given by the formula
$$\mathfrak{sp}(d,\mathbb{R})=\Big\{A\in M_d(\mathbb{R}):\ \Omega A+A^*\Omega=0\Big\}.$$
Define a representation $\pi$ of the Lie algebra $\mathfrak{sp}(d,\mathbb{R})$ on ${\rm Poly}(\Sb^{d-1})$ by the formula
\begin{equation}\label{pi explicit}
(\pi(A)b)(t)=\langle (\nabla_{\mathbb{S}^{d-1}} b)(t),At\rangle-d\langle At,t\rangle b(t).
\end{equation}

In Lemmas \ref{first reduction lemma} and \ref{second reduction lemma} below, we need the following explicit expression for the spherical gradient:
\begin{equation}\label{spherical gradient explicit}
(\nabla_{\mathbb{S}^{d-1}} b)(t)=(\nabla b)(t)-\langle (\nabla b)(t),t\rangle t,\quad b\in C^{\infty}(\mathbb{R}^d).
\end{equation}

The significance of the map $\pi$ (as well as the reason for it being a representation) is given in the following:

\begin{lem}\label{differentiation of V}
If $b\in\mathrm{Poly}(\Sb^{d-1})$ and if $A\in\mathfrak{sp}(d,\Rl),$ then $e^{sA}\in {\rm Sp}(d,\mathbb{R})$ and
$$\frac{d}{ds}(V_{e^{sA}}b)(t)|_{s=0} =(\pi(A)b)(t).$$
\end{lem}
\begin{proof} For every $t\in\mathbb{S}^{d-1},$ we have
$$e^{sA}t=t+s\cdot At+o(s),\quad s\to0.$$
Thus,
$$|e^{sA}t|^2=\langle e^{sA}t,e^{sA}t\rangle=\langle t+s\cdot At,t+s\cdot At\rangle+o(s)=1+2s\cdot \langle At,t\rangle+o(s),\quad s\to0.$$
Thus,
$$|e^{sA}t|^{-1}=1-s\cdot \langle At,t\rangle+o(s),\quad s\to0$$
$$\frac{e^{sA}t}{|e^{sA}t|}=t+s\cdot(At-t\langle At,t\rangle)+o(s),\quad s\to0.$$
Since $b$ is smooth, it follows that
$$b(\frac{e^{sA}t}{|e^{sA}t|})=b(t)+s\cdot\langle(\nabla_{\mathbb{S}^{d-1}} b)(t),At-t\langle At,t\rangle\rangle+o(s),\quad s\to0.$$
Since spherical gradient at the point $t$ is orthogonal to $t,$ it follows that
$$b(\frac{e^{sA}t}{|e^{sA}t|})=b(t)+s\cdot\langle(\nabla_{\mathbb{S}^{d-1}} b)(t),At\rangle+o(s),\quad s\to0.$$
Multiplying the latter equality with
$$|e^{sA}t|^{-d}=1-ds\cdot \langle At,t\rangle+o(s),\quad s\to0,$$
we obtain
$$(V_{e^{sA}}b)(t)=b(t)+s(\pi(A)b)(t)+o(s),\quad s\to0.$$
This completes the proof.
\end{proof}

\begin{lem}\label{pass to lie algebra} 
If $l\in C(\Sb^{d-1})^*$ is such that $l\circ V_g=l,$ $g\in {\rm Sp}(d,\mathbb{R}),$ then $l(\pi(A)b)=0$ for every $b\in {\rm Poly}(\Sb^{d-1})$ and for every $A\in \mathfrak{sp}(d,\mathbb{R}).$
\end{lem}
\begin{proof} Let $b \in \mathrm{Poly}(\Sb^{d-1})$.
Since polynomials are smooth, from Lemma \ref{differentiation of V} it follows that
$$\Big\|V_{e^{sA}}b-b-s\cdot\pi(A)b\Big\|_{\infty}=o(s),\quad s\to0.$$
Additionally since $l$ is continuous, we obtain
$$l(V_{e^{sA}}b)-l(b)-s\cdot l(\pi(A)b)=o(s),\quad s\to0.$$
By assumption, $l(V_{e^{sA}}b)=l(b).$ Thus,
$$s\cdot l(\pi(A)b)=o(s),\quad s\to0.$$
This completes the proof.
\end{proof}

\begin{lem}\label{first reduction lemma} Let $l\in C(\Sb^{d-1})^*$ be such that $l\circ V_g=l,$ $g\in {\rm Sp}(d,\mathbb{R}).$ For every ${\bf n}\in\mathbb{Z}^d_+,$ we have
$$l(b_{{\bf n}+2e_{2k-1}}) = \frac{n_{2k-1}+1}{n_{2k}+1}l(b_{{\bf n}+2e_{2k}}),\quad 1\leq k\leq\frac{d}{2}.$$
\end{lem}
\begin{proof} Let
$$
A=0^{\oplus (2k-2)}\oplus
\begin{pmatrix}
0&1\\
-1&0
\end{pmatrix}
\oplus 0^{\oplus (d-2k)}.
$$
It is clear that $A$ commutes with $\Omega$, and since $A^*=-A,$ it follows that $A\Omega + \Omega A^* = 0,$ so $A\in\mathfrak{sp}(d,\mathbb{R}).$ We consider $\pi(A)b_{{\bf n}+e_{2k-1}+e_{2k}}(t).$ Taking into account that $A$ is antisymmetric, we replace the spherical gradient in \eqref{pi explicit} with the usual gradient one and write
$$\pi(A)b_{\bf n+e_{2k-1}+e_{2k}}(t)=\langle\nabla b_{{\bf n}+e_{2k-1}+e_{2k}}(t),At\rangle.$$
Thus,
\begin{align*}
\pi(A)b_{\bf n+e_{2k-1}+e_{2k}}(t) &= (n_{2k-1}+1)t_{2k}b_{{\bf n}+e_{2k}}(t)-(n_{2k}+1)t_{2k-1}b_{{\bf n}+e_{2k-1}}(t)\\
&= (n_{2k-1}+1)b_{{\bf n}+2e_{2k}}(t)-(n_{2k}+1)b_{{\bf n}+2e_{2k-1}}(t).
\end{align*}
Applying $l$ and using Lemma \ref{pass to lie algebra}, we conclude the argument.
\end{proof}

\begin{lem}\label{second reduction lemma} Let $l\in C(\Sb^{d-1})^*$ be such that $l\circ V_g=l,$ $g\in {\rm Sp}(d,\mathbb{R}).$ For every ${\bf n}\in\mathbb{N}^d,$ we have
\begin{equation}\label{main red}
l(b_{{\bf n}+2e_k})=\frac{n_k+1}{\|{\bf n}\|_1+d}l(b_{{\bf n}}),\quad 1\leq k\leq d.
\end{equation}
\end{lem}
\begin{proof} Set
$$B=0^{\oplus (2k-2)}\oplus
\begin{pmatrix}
1&0\\
0&-1
\end{pmatrix}
\oplus 0^{\oplus (d-2k)}.
$$
It is easy to see that $B\Omega = -\Omega B$, so $B\in\mathfrak{sp}(d,\mathbb{R}).$ It is now clear from \eqref{spherical gradient explicit} that
$$(\nabla_{\mathbb{S}^{d-1}} b_{{\bf n}})(t)=(\nabla b_{{\bf n}})(t)-\|{\bf n}\|_1b_{{\bf n}}(t)t.$$
Substituting this into \eqref{pi explicit}, we obtain
$$(\pi(B)b_{\bf n})(t)=(n_{2k-1}-n_{2k})b_{{\bf n}}(t)-(d+\|{\bf n}\|_1)(t_{2k-1}^2-t_{2k}^2)b_{{\bf n}}(t).$$
In other words,
$$\pi(B)b_{\bf n}=(n_{2k-1}-n_{2k})b_{{\bf n}}-(d+\|{\bf n}\|_1)(b_{{\bf n}+2e_{2k-1}}-b_{{\bf n}+2e_{2k}}).$$

Applying $l$ and using Lemma \ref{pass to lie algebra}, we obtain
\begin{equation}\label{red2}
(n_{2k-1}-n_{2k})l(b_{{\bf n}})-(d+\|{\bf n}\|_1)(l(b_{{\bf n}+2e_{2k-1}})-l(b_{{\bf n}+2e_{2k}}))=0.
\end{equation}

Suppose first that $n_{2k-1}\neq n_{2k}.$ In this case, it follows from Lemma \ref{first reduction lemma} that
$$(n_{2k-1}-n_{2k})l(b_{\bf n})=(d+\|{\bf n}\|_1)\left(\frac{n_{2k-1}+1}{n_{2k}+1}-1\right)l(b_{{\bf n}+2e_{2k}}).$$
Since $n_{2k-1}\neq n_{2k},$ it follows that
\begin{equation}\label{even case red}
l(b_{\bf n})=\frac{d+\|{\bf n}\|_1}{n_{2k}+1}l(b_{{\bf n}+2e_{2k}}).
\end{equation}
Using Lemma \ref{first reduction lemma}, we obtain that
\begin{equation}\label{odd case red}
l(b_{\bf n})=\frac{d+\|{\bf n}\|_1}{n_{2k-1}+1}l(b_{{\bf n}+2e_{2k-1}}).
\end{equation}
A combination of \eqref{even case red} and \eqref{odd case red} yields the assertion for the case when $n_{2k-1}\neq n_{2k}.$

The final case to consider is $n_{2k} = n_{2k-1}.$ By assumption, $n_{2k-1},n_{2k}>0$ and we may exclude the case $n_{2k-1}=n_{2k}=1$ by Lemma \ref{odd lemma}. Hence assume $n_{2k-1}=n_{2k}\geq 2.$ Let ${\bf m}:={\bf n}+2e_{2k-1}-2e_{2k}.$ Then, $m_{2k} \neq m_{2k-1},$ so by the previous case,
\begin{equation}\label{even case red m}
l(b_{{\bf m}+2e_{2k}})=\frac{m_{2k}+1}{d+\|{\bf m}\|_1}l(b_{\bf m}).
\end{equation}
Now, let ${\bf p}={\bf n}-2e_{2k}$ so that ${\bf n}={\bf p}+2e_{2k}$ and ${\bf m}={\bf p}+2e_{2k-1}.$ By Lemma \ref{first reduction lemma}, we have
\begin{equation}\label{m to n}
l(b_{\bf m})=l(b_{{\bf p}+2e_{2k-1}})=\frac{p_{2k-1}+1}{p_{2k}+1}l(b_{{\bf p}+2e_{2k}})=\frac{n_{2k-1}+1}{n_{2k}-1}l(b_{{\bf n}}).
\end{equation}
Recalling that ${\bf n}+2e_{2k-1}={\bf m}+2e_{2k},$ we obtain,
$$l(b_{{\bf n}+2e_{2k-1}})=l(b_{{\bf m}+2e_{2k}})\stackrel{\eqref{even case red m}}{=}\frac{m_{2k}+1}{d+\|{\bf m}\|_1}l(b_{\bf m})\stackrel{\eqref{m to n}}{=}\frac{n_{2k}-1}{d+\|{\bf n}\|_1}\cdot \frac{n_{2k-1}+1}{n_{2k}-1}l(b_{{\bf n}}).$$
This equality is exactly \eqref{odd case red} for the case when $n_{2k}=n_{2k-1}.$ A similar argument delivers \eqref{even case red} for the case when $n_{2k}=n_{2k-1}.$ This proves the assertion for the case when $n_{2k}=n_{2k-1}.$
\end{proof}

We now have all the results required to prove Theorem \ref{main symplectic thm}.

\begin{proof}[Proof of Theorem \ref{main symplectic thm}] Without loss of generality, we may assume that $l(b_{{\bf n}^0})=0,$ where ${\bf n}^0=(2,\cdots,2)$ (otherwise we consider $l-{\rm const}\cdot m$). Let ${\bf n}\in\mathbb{N}^d.$ If some $n_k$ is odd, then $l(b_{{\bf n}})=0$ by Lemma \ref{odd lemma}. If every $n_k$ is even, then $l(b_{{\bf n}})=0$ by the assumption and Lemma \ref{second reduction lemma}.

Let $b\in C^{\infty}(\mathbb{S}^{d-1})$ vanish on every equator $\{t_k=0\}.$ We write
$$b(t)=(\prod_{k=1}^dt_k)\cdot h(t),\quad h\in C(\mathbb{S}^{d-1}).$$
Since $h$ can be approximated by polynomials, it follows that $b$ can be approximated by a linear combination of $b_{{\bf n}},$ ${\bf n}\in\mathbb{N}^d.$ By the preceding paragraph, $l(b)=0.$

Let the function $b\in C(\mathbb{S}^{d-1})$ vanish on every equator $\{t_k=0\}.$ By the preceding paragraph and continuity of $l,$ we have that $l(b)=0.$ It follows from the Riesz theorem that $l$ is the measure $\nu$ supported on the union of all equators.

For every $g\in g\in {\rm Sp}(d,\mathbb{R}),$ let $B_g:\mathbb{S}^{d-1}\to\mathbb{S}^{d-1}$ be given by the formula $t\to \frac{g^{-1}t}{|g^{-1}t|}.$ Since $l=l\circ V_g,$ it follows that $\nu\circ B_g$ is also supported on the union of all equators. In other words, $\nu$ is supported on the set
$$\bigcup_{k=1}^d\{(gt)_k=0\}.$$
Recall that the intersection of $d$ generic hyperplanes is $\{0\}$ (which, obviously, does not belong to the sphere). Thus, one can choose a finite collection $\{g_i\}_{i\in\mathbb{I}}\subset{\rm Sp}(d,\mathbb{R})$ such that
$$\bigcap_{i\in\mathbb{I}}(\bigcup_{k=1}^d\{(gt)_k=0\})=\varnothing.$$
Hence, $\nu$ is nowhere supported and, therefore, $l=0.$
\end{proof}

\section{$q(\mathcal{A}_2)$ is $C(\Sb^2)$}\label{identification of sphere algebra}

Let $q:\mathcal{L}(L_2({\rm SU}(2)))\to\mathcal{Q}(L_2({\rm SU}(2)))$ be the canonical quotient map. Recall
that we define $b_k = \frac{D_k}{\sqrt{-\Delta}}$, set $b_k(1) = \frac{1}{\sqrt{3}}$.
Since for each $j,k = 1,2,3$ we have that $[b_j,b_k]$ is compact, it follows that $[q(b_j),q(b_k)] = 0$,
so the algebra $q(\A_2)$ is commutative, and generated by three commuting self-adjoint
elements $q(b_1),q(b_2),q(b_3)$ satisfying $q(b_1)^2+q(b_2)^2+q(b_3)^2 = 1$.

The algebra $C(\Sb^2)$ is the universal $C^*-$algebra generated by self-adjoint commuting elements $t_k,$ $1\leq k\leq 3,$ satisfying the condition $t_1^2+t_2^2+t_3^2=1.$ Hence, the mapping $t_k\to q(b_k),$ $1\leq k\leq 3,$ extends to a surjective $*-$homomorphism $u:C(\Sb^2)\to q(\A_2).$
This section is devoted to the proof of the following:
\begin{thm}\label{isometry thm} The mapping $u:C(\Sb^2)\to q(\A_2)$ is an isometric $*$-isomorphism.
\end{thm}

We introduce a mapping $\Ec$, which is defined
to be the conditional expectation operator from $\Bc(L_2(\SU(2)))$ to the sub-algebra
generated by the joint spectral projections of $D_1$ and $\Delta$.
The operator $\Delta$ has compact resolvent, and commutes with $D_1$. Hence
there is a sequence $\{p_k\}_{k=0}^\infty$ of finite rank projections onto the eigenspaces for $D_1$ and $\Delta$.
We may write $\Ec$ in terms of the sequence $\{p_k\}_{k=0}^\infty$,
\begin{equation*}
    \Ec(T) = \sum_{k=0}^\infty p_kTp_k
\end{equation*}
where the sum is weakly convergent. Since the ranges of each $p_k$
are mutually orthogonal, we also have $\|\Ec(T)\| \leq \|T\|$. 
Since the spectral projections of $\Delta$ span all of $L_2(\SU(2))$, 
it follows that $\Ec$ is trace preserving. Thus $T$ maps rank one operators to trace class operators, and so by
continuity $T$ maps compact operators to compact operators.
%
%

Since $\Ec$ is the conditional expectation onto the algebra generated by the joint spectral projections
of $D_1$ and $\Delta$, we have that $\Ec(b_1T) = b_1\Ec(T)$ and $\Ec(Tb_1) = \Ec(T)b_1$.

There is a canonical surjective group homomorphism $\eta:\SU(2)\to\SO(3)$ defined as follows.
\begin{equation}\label{eta def}
(\eta(g))_{k,j} = \frac{1}{2}\tr(g\sigma_kg^*\sigma_j).
\end{equation}

For example, we have
\begin{equation}\label{sigma_1 matrix}
\eta(e^{it\sigma_1})=
\begin{pmatrix}
1 & 0 & 0\\
0 & \cos(2t) & \sin(2t)\\
0 & -\sin(2t) & \cos(2t)
\end{pmatrix}.
\end{equation}

The relationship between $\eta$ and the operators $D_j$ is detailed in the following:
\begin{lem}\label{g conjugation}\label{eta transformation}
Let $j = 1,2,3$. We have:
\begin{equation*}
\lambda_l(g)D_j\lambda_l(g)^{-1} = \sum_{k=1}^3 (\eta(g))_{k,j}D_k.
\end{equation*}
Furthermore, for any $f\in C(\mathbb{S}^2),$ we have
\begin{equation*}
\lambda_l(g)u(f)\lambda_l(g)^{-1} = u(f\circ \eta(g)).
\end{equation*} 
\end{lem}
\begin{proof} For $j = 1,2,3$,
\begin{equation}\label{matrix elements}
g\sigma_jg^* = \sum_{k=1}^3 (\eta(g))_{k,j}\sigma_k.
\end{equation}

By definition, given $\xi \in C^\infty(\SU(2))$, $g,h \in \SU(2)$ and $j = 1,2,3$ we have
\begin{equation*}
    (\lambda_l(g)D_j\lambda_l(g)^{-1}\xi)(h) =-\langle \nabla \xi(h),g\sigma_jg^*h\rangle
\end{equation*}
Here $\nabla$ denotes above the gradient on $\SU(2)$. Now applying \eqref{matrix elements},
\begin{equation*}
(\lambda_l(g)D_j\lambda_l(g)^{-1}\xi)(h) =-\sum_{k=1}^3 (\eta(g))_{k,j}\langle \nabla \xi(h),\sigma_k\rangle.
\end{equation*} 
Thus
\begin{equation*}
    \lambda_l(g)D_j\lambda_l(g)^{-1} = \sum_{k=1}^3 (\eta(g))_{k,j}D_{k}.
\end{equation*}
This prove the first part of the Lemma.

Since $\Delta$ commutes with each $D_j$ and $\lambda_l(g)$, 
\begin{equation}\label{eta bj}
    \lambda_l(g)b_j\lambda_l(g)^{-1} = \sum_{k=1}^3 (\eta(g))_{k,j}b_{k}.
\end{equation}

Consider now the continuous $*-$homomorphisms from $C(\mathbb{S}^2)$ to $q(\mathcal{A}_2)$ defined by the formulae
\begin{equation*}
\varpi_1:f \mapsto \lambda_l(g)u(f)\lambda_l(g)^{-1},\quad \varpi_2:f\mapsto u(f\circ \eta(g))
\end{equation*}
By \eqref{eta bj}, we have $\varpi_1(t_j)=\varpi_2(t_j),$ $j=1,2,3.$ Since functions $t\to t_j,$ $j = 1,2,3$ generate $C(\Sb^2)$, the result follows.
\end{proof}

The following Lemma is used in the proof of Theorem \ref{isometry thm}.
\begin{lem}\label{sphere motivation} Let $\mathbb{E}:L_{\infty}(\Sb^2)\to L_{\infty}(\Sb^2)$ be the conditional expectation onto the subalgebra generated by $t_1.$ We have
\begin{enumerate}
\item If $n_2$ or $n_3$ is odd, then
$$\mathbb{E}(t_2^{n_2}t_3^{n_3})=0.$$
\item If both $n_2$ and $n_3$ are even, then
$$\mathbb{E}(t_2^{n_2}t_3^{n_3})=\frac1{\pi}B\left(\frac{n_2+1}{2},\frac{n_3+1}{2}\right)(1-t_1^2)^{\frac{n_2+n_3}{2}}.$$
Here $B$ is the Beta function.
\end{enumerate}
\end{lem}
\begin{proof}
We compute the integrals,
\begin{equation*}
   \int_{\Sb^2} t_1^{n_1}t_2^{n_2}t_3^{n_3}\,dt,\quad n_1,n_2,n_3 \geq 0.
\end{equation*}
It is easy to see by reflecting around the $t_2$ or $t_3$ coordinate that this integral vanishes when one of $\{n_2,n_3\}$ is odd. This proves the first claim.
%

%
 To see the second claim, we pass to the spherical coordinates:
 $$t_1=\cos(\theta),\quad t_2=\sin(\theta)\sin(\phi),\quad t_3=\sin(\theta)\cos(\phi).$$
 For $\phi\in(0,2\pi),$ $\theta\in(0,\pi)$ and we have $dt=\sin(\theta)d\theta d\phi.$
 Thus,
 \begin{align*}
     \int_{\Sb^2}t_1^{n_1}t_2^{n_2}t_3^{n_3}dt &= \int_0^{\pi}\int_0^{2\pi}\cos^{n_1}(\theta)\sin^{n_2+n_3+1}(\theta)\sin^{n_2}(\phi)\cos^{n_3}(\phi)d\phi d\theta\\
                                               &= \int_0^{\pi}\cos^{n_1}(\theta)\sin^{n_2+n_3+1}(\theta)d\theta\cdot \int_0^{2\pi}\sin^{n_2}(\phi)\cos^{n_3}(\phi)d\phi.
 \end{align*}
 
 For the second part, assume that both $n_2$ and $n_3$ are even. We have that:
 $$\int_0^{2\pi}\sin^{n_2}(\phi)\cos^{n_3}(\phi)d\phi=2B\left(\frac{n_2+1}{2},\frac{n_3+1}{2}\right).$$
 (see, for example, \cite[6.2.1]{Abramowitz-Stegun}).
 Thus,
 \begin{align*}
    \int_{\Sb^2}t_1^{n_1}t_2^{n_2}t_3^{n_3}dt &= \frac1{\pi}B(\frac{n_2+1}{2},\frac{n_3+1}{2})\cdot\int_0^{\pi}\int_0^{2\pi}\cos^{n_1}(\theta)\sin^{n_2+n_3+1}(\theta)d\phi d\theta\\
                                              &= \frac1{\pi}B(\frac{n_2+1}{2},\frac{n_3+1}{2})\int_{\Sb^2}t_1^{n_1}(1-t_1^2)^{\frac{n_2+n_3}{2}}dt.
 \end{align*}
 This proves the claim.
 %
 %
\end{proof}

\begin{lem}\label{odd computational lemma} 
We have:
\begin{enumerate}
\item If $n_2$ or $n_3$ is odd, then
$$q(\Ec(b_1^{n_1}b_2^{n_2}b_3^{n_3})) = 0.$$
\item If both $n_2$ and $n_3$ are even, then
$$q(\Ec(b_1^{n_1}b_2^{n_2}b_3^{n_3})) =\frac1{\pi}B\left(\frac{n_2+1}{2},\frac{n_3+1}{2}\right)q(b_1)^{n_1}(1-q(b_1)^2)^{\frac{n_2+n_3}{2}}.$$
Here $B$ is the Beta function.
\end{enumerate}
\end{lem}
\begin{proof} Since $\Ec(b_1^{n_1}T) = b_1^{n_1}\Ec(T)$ for any bounded operator $T$,
it suffices to prove the result for $n_1 = 0$.

Let $n \geq 0$, and let $t \in \Rl$.
From Lemma \ref{g conjugation} and \eqref{sigma_1 matrix},
\begin{align*}
    e^{itD_1}b_3e^{-itD_1} &= \sin(2t)b_2+\cos(2t)b_3
\end{align*}
Since $\Ec(b_3^n)$ contains the algebra generated by the spectral projections of $D_1$, we have:
\begin{equation*}
    \Ec(b_3^n) = e^{itD_1}\Ec(b_3^n)e^{-itD_1}.
\end{equation*}
So,
\begin{align*}
    \Ec(b_3^n) &= \Ec(e^{itD_1}b_3^ne^{-itD_1})\\
               &= \Ec((e^{itD_1}b_3e^{-itD_1})^n)\\
               &= \Ec((\sin(2t)b_2+\cos(2t)b_3)^n).
\end{align*}
However $[b_2,b_3] \in \Kc(L_2(\SU(2)))$, so $(q\Ec)$ vanishes on the ideal generated by $[b_2,b_3]$. Hence,

\begin{equation}\label{positive t}
    (q\mathcal{E})(b_3^n) = \sum_{l=0}^n\binom{n}{l}\sin^l(2t)\cos^{n-l}(2t)(q \Ec)(b_2^lb_3^{n-l}).
\end{equation}

We now prove the first assertion of the lemma. Substituting $-t$ instead of $t,$ we obtain
\begin{equation}\label{negative t}
    (q\Ec)(b_3^n) = \sum_{l=0}^n(-1)^l\binom{n}{l}\sin^l(2t)\cos^{n-l}(2t)(q\Ec)(b_2^lb_3^{n-l}).
\end{equation}
Therefore by subtracting \eqref{negative t} from \eqref{positive t}, we have
\begin{equation}\label{odd linear dependence}
    \sum_{\substack{0\leq l\leq n\\ l\mbox{ is odd}}}\binom{n}{l}\sin^l(2t)\cos^{n-l}(2t)(q\Ec)(b_2^lb_3^{n-l}) = 0.
\end{equation}

The functions $\sin^l(2t)\cos^{n-l}(2t)$ are linearly independent, and so it follows that $(q\Ec)(b_2^{l}b_3^{n-l}) = 0$
when $l$ is odd. 
Substituting $\frac{\pi}{2}-t$ into \eqref{positive t}, we obtain by an identical argument that $(q\Ec)(b_2^{n-l}b_3^l) = 0$ when $l$ is odd. This
proves the first part of the lemma.

Now we prove the second part of the lemma. Assume now that $n$ is even.
Since $\sin^2(2t)+\cos^2(2t) = 1$, we get:
\begin{align*}
    1 &= (\cos^2(2t)+\sin^2(2t))^{n/2}\\
      &= \sum_{j=0}^{n/2} \binom{n/2}{j} \cos^{n}(2t)\sin^{n-2j}(2t).
\end{align*}

By changing the index of summation, it follows that
\begin{equation*}
    \sum_{\substack{0\leq l\leq n\\ l\mbox{ is even}}}\binom{\frac{n}{2}}{\frac{l}{2}}\sin^l(2t)\cos^{n-l}(2t)=1.
\end{equation*}
Multiplying the left hand side of \eqref{odd linear dependence} by $1$, we get
\begin{align*}
\sum_{\substack{0\leq l\leq n\\ l\mbox{ is even}}}\binom{\frac{n}{2}}{\frac{l}{2}}\sin^l(2t)\cos^{n-l}(2t)(q\Ec)(b_3^n) = \sum_{l=0}^n\binom{n}{l}\sin^l(2t)\cos^{n-l}(2t)(q\Ec)(b_2^lb_3^{n-l}).
\end{align*}

Since the functions $t\to \sin^l(2t)\cos^{n-l}(2t),$ $0\leq l\leq n,$ are linearly independent, it follows that
\begin{equation}\label{b_3 in terms of b_2}
    \binom{\frac{n}{2}}{\frac{l}{2}}(q\Ec)(b_3^n) =  \binom{n}{l}(q\Ec)(b_2^lb_3^{n-l}),\quad l\mbox{ is even}.
\end{equation}

In particular for $l=2$ we have
$$\mathcal{E}(b_2^2b_3^{n-2}) = \frac1{n-1}(q\Ec)(b_3^n).$$
Since $1-b_1^2=b_2^2+b_3^2$, we have
\begin{align*}
    (1-q(b_1)^2)(q\Ec)(b_3^{n-2}) &= (q\Ec)((b_2^2+b_3^2)b_3^{n-2})\\
                                    &= (q\Ec)(b_3^n)+\frac{1}{n-1}(q\Ec)(b_3^n)\\ 
                                    &= \frac{n}{n-1}(q\Ec)(b_3^n).
\end{align*}
So:
\begin{equation*}
    (q\Ec)(b_3^n) = \frac{n-1}{n}q(\Ec)(b_3^{n-2})(1-q(b_1)^2).
\end{equation*}
By induction, we have
\begin{equation}\label{b_3 formula}
    (q\Ec)(b_3^n) = \frac{(n-1)!!}{n!!}(1-q(b_1)^2)^{\frac{n}{2}}.
\end{equation}
Substituting \eqref{b_3 formula} into \eqref{b_3 in terms of b_2}, we obtain
$$(q\Ec)(b_2^lb_3^{n-l}) = \frac{\binom{\frac{n}{2}}{\frac{l}{2}}}{\binom{n}{l}}\cdot\frac{(n-1)!!}{n!!}(1-q(b_1)^2)^{\frac{n}{2}}.$$

We also compute the Beta function:
\begin{equation*}
B\left(\frac{l+1}{2},\frac{n-l+1}{2}\right) = \frac{\Gamma\left(\frac{l}{2}+\frac{1}{2}\right)\Gamma\left(\frac{n-l}{2}+\frac{1}{2}\right)}{\Gamma(\frac{n}{2}+1)}.
\end{equation*}
Taking into account that, for even $n$ and $l$
$$\Gamma\left(\frac{l}{2}+\frac{1}{2}\right)=\frac{(l-1)!!}{2^{\frac{l}{2}}}\pi^{\frac12},\quad \Gamma\left(\frac{n-l}{2}+\frac{1}{2}\right)=\frac{(n-l-1)!!}{2^{\frac{n-l}{2}}}\pi^{\frac12}$$
it follows that
$$\frac1{\pi}B\left(\frac{l+1}{2},\frac{n-l+1}{2}\right) = \frac{l!(n-l)!}{2^{\frac{n}{2}}\cdot l!!\cdot (n-l)!!\cdot(\frac{n}{2})!}=\frac{\binom{\frac{n}{2}}{\frac{l}{2}}}{\binom{n}{l}}\cdot\frac{(n-1)!!}{n!!}.$$

Taking $l = n_2$ and $n = n_2+n_3$, we finally have:
\begin{equation*}
    (q\Ec)(b_2^{n_2}b_3^{n_3}) = \frac{1}{\pi}B\left(\frac{n_2+1}{2},\frac{n_3+1}{2}\right)(1-q(b_1)^2)^{\frac{n_2+n_3}{2}}.
\end{equation*}

Taking $l = n_2$ and $n = n_2+n_3$, we finally have:
\begin{equation*}
    (q\Ec)(b_2^{n_2}b_3^{n_3}) = \frac{1}{\pi}B\left(\frac{n_2+1}{2},\frac{n_3+1}{2}\right)(1-q(b_1)^2)^{\frac{n_2+n_3}{2}}.
\end{equation*}
\end{proof}

\begin{cor}\label{u intertwining}
    Let $f \in C(\Sb^2)$. Then for any $x \in \A_2$ with $u(f) = q(x)$, we have:
    \begin{equation*}
        q(\Ec(x)) = u(\mathbb{E}(f))
    \end{equation*}
    where $\mathbb{E}$ is the conditional expectation onto the subalgebra generated by $t_1$
    as in Lemma \ref{sphere motivation}.
\end{cor}
\begin{proof}
    Since $u(t_1) = q(b_1)$, by applying $u$ to the result of Lemma \ref{sphere motivation} when $n_2$ and $n_3$ are even:
    \begin{equation*}
        u(\mathbb{E}(t_1^{n_1}t_2^{n_2}t_3^{n_3})) = \frac{1}{\pi}B\left(\frac{n_2+1}{2},\frac{n_3+1}{2}\right)q(b_1)^{n_1}(1-q(b_1)^2)^{\frac{n_2+n_3}{2}}
    \end{equation*}
    However due to Lemma \ref{odd computational lemma}, the above hand side above is exactly $(q\Ec)(b_1^{n_1}b_2^{n_2}b_3^{n_3})$. Hence
    when both $n_2$ and $n_3$ are even:
    \begin{equation*}
        u(\mathbb{E}(t_1^{n_1}t_2^{n_2}t_3^{n_3})) = (q\Ec)(b_1^{n_1}b_2^{n_2}b_3^{n_3}).
    \end{equation*}
    The above also holds when one of $n_2$ or $n_3$ is odd, since both sides are zero from Lemmas \ref{sphere motivation} and \ref{odd computational lemma}.
    Hence the result is proved for $f(t) = t_1^{n_1}t_2^{n_2}t_3^{n_3}$. So by linearity, the result follows for all polynomials $f$.
    
    As all of the maps $q$, $\Ec$, $u$ and $\mathbb{E}$ are continuous, the result then follows for all $f \in C(\Sb^2)$.
\end{proof}

\begin{lem}\label{sphere key estimate} If $f\in C(\Sb^2)$, then
$$\|u(f)\|\geq |f(1,0,0)|.$$
\end{lem}
\begin{proof} 
 
    Suppose first that $f$ depends only on the first variable, so $f(t)=p(t_1)$ for some $p \in C([-1,1]).$ Since $q(\A_2)$ has the Calkin algebra norm, we have
    $$\|u(f)\|=\inf_{K\in\mathcal{K}(L_2({\rm SU}(2)))}\|p(b_1)+K\|.$$
    It is shown in \cite[Theorem 11.9.3]{ruzh-turu} that there exists a decomposition of $L_2(\SU(2))$ into orthogonal finite dimensional subspaces
    $$\{V_{l,m}\;:\;l \in \frac{1}{2}\Ntrl, m = -l,-l+1,\ldots,0,1,\ldots,l\}$$
    such that if $e_{l,m} \in V$, then (note that $D_1$ is denoted $\partial_0$ in \cite{ruzh-turu})
    $$\Delta e_{l,m} = -l(l+1)e_{l,m},\quad D_1e_{l,m} = me_{l,m},\quad -l\leq m\leq l,\quad l\geq0.$$
    Hence, $b_1e_{l,m} = \frac{m}{\sqrt{l(l+1)}}e_{l,m}$ when $l \neq 0$, so it follows that the eigenvalues of $b_1$ are dense in $[-1,1]$. Therefore,
    $$\inf_{K \in \Kc(L_2(\SU(2)))}\|p(b_1)+K\|=\|p\|_\infty.$$
    So, 
    $$\|u(f)\|\geq\|p\|_{\infty} = \|f\|_\infty.$$ So the statement of the Lemma is proved in the case when $f(t)$ depends only on $t_1$.

    Now we prove the general statement. Since $\mathbb{E}(f)$ depends only on $t_1$, by the preceding argument:
 \begin{equation*}
    \|u(\mathbb{E}(f))\| \geq \|\mathbb{E}(f)\|_\infty.
 \end{equation*}
 Since $u$ is surjective, we may choose $x \in \A_2$ such that $q(x) = u(f)$. So from Lemma
 \ref{u intertwining},
 \begin{equation*}
    \|q(\Ec(x))\| = \|u(\mathbb{E}(f))\| \geq \|\mathbb{E}(f)\|_\infty.
 \end{equation*}
 Since $\Ec$ maps $\Kc(L_2(\SU(2)))$ to itself,
 \begin{align*}
    \|q(\Ec(x))\| &= \inf_{K \in \Kc(L_2(\SU(2)))}\|\Ec(x)+K\|\\
                  &\leq \inf_{K \in \Kc(L_2(\SU(2)))} \|\Ec(x)+\Ec(K)\|\\
                  &\leq \inf_{K \in \Kc(L_2(\SU(2)))} \|x+K\|\\
                  &= \|q(x)\|.
 \end{align*}
 Therefore, $\|q(\Ec(x))\| \leq \|q(x)\| = \|u(f)\|$, it follows that
 \begin{equation*}
    \|u(f)\| \geq \|\mathbb{E}(f)\|_\infty.
 \end{equation*} 
 Since $\|\mathbb{E}(x)\|_\infty \geq |x(1,0,0)|$, this yields the result.
\end{proof}

\begin{proof}[Proof of Theorem \ref{isometry thm}] 
    Since we know that the mapping $u$ is surjective and a $*$-homomorphism,
    it suffices to show that $u$ is an isometry. We know that $\|u(f)\| \leq \|f\|_\infty$, 
    so it suffices to prove the reverse inequality.
    
    From Lemma \ref{eta transformation},
    \begin{equation*}
        \lambda_l(g)u(f)\lambda_l(g)^{-1} = u(f\circ \eta(g)).
    \end{equation*}
    So $\|u(f)\| = \|u(f\circ \eta(g))\|$. Applying now Lemma \ref{sphere key estimate},
    \begin{equation*}
        \|u(f)\| \geq |(f\circ \eta(g))(1,0,0)|.
    \end{equation*}
    Since the map $\eta$ is surjective, we have that $\eta(g)$ is an arbitrary element
    of $\SO(3)$. Thus $\|u(f)\| \geq \|f\|_\infty$, as required.
\end{proof}

\end{document}